\documentclass[final,1p,times]{elsarticle}
\usepackage{amsmath}
\usepackage{verbatim}  
\allowdisplaybreaks[4]
\usepackage{amsthm}
\usepackage{amscd}
\RequirePackage{CJKnumb}
\usepackage{amsfonts}
\usepackage{amssymb}
\usepackage{graphicx}
\usepackage{caption} 
\usepackage{subfigure}
\usepackage{subcaption}
\usepackage{booktabs}
\biboptions{sort&compress}
\usepackage{xcolor}
\usepackage[pagebackref, colorlinks, citecolor=green, linkcolor=red, urlcolor=blue]{hyperref}
\usepackage{pgfplots}
\pgfplotsset{compat=newest} 
\numberwithin{equation}{section}
\newtheorem{theorem}{Theorem}[section]
\usepackage{mathrsfs}
\usepackage{titletoc}
\usepackage{float}
\usepackage{longtable}
\usepackage{adjustbox}
\theoremstyle{remark} 

\journal{***}
\begin{document}
	\begin{frontmatter}
		\title{A cytokine-enhanced viral infection model with CTL immune response, distributed delay and saturation incidence}
		 \author{Xiaodong Cao} 
		\ead{xiaodongcao@cornell.edu}
		\address{Department of Mathematics, Cornell University, Ithaca, NY 14853}
  \author{Songbo Hou \corref{cor1}}
		\ead{housb@cau.edu.cn}
		\address{Department of Applied Mathematics, College of Science, China Agricultural University,  Beijing, 100083, P.R. China}
		\author{Xiaoqing Kong}
		\ead{kxq@cau.edu.cn}
		\address{Department of Applied Mathematics, College of Science, China Agricultural University,  Beijing, 100083, P.R. China}
		
		\cortext[cor1]{Corresponding author: Songbo Hou}
		\begin{abstract}
			In this paper, we propose a delayed cytokine-enhanced viral infection model incorporating saturation incidence and immune response. We compute the basic reproduction numbers and  introduce a convex cone to discuss the impact of non-negative initial data on solutions. By defining appropriate Lyapunov functionals and employing LaSalle's invariance principle, we investigate the stability of three equilibria: the disease-free equilibrium, the immunity-inactivated equilibrium, and the immunity-activated equilibrium. We establish conditions under which these equilibria are globally asymptotically stable. Numerical analyses not only corroborate the theoretical results but also reveal that intervention in virus infection can be achieved by extending the delay period.
			
		\end{abstract}	
	\end{frontmatter}
	
	\section{Introduction}

	HIV is widely recognized as the causative agent of AIDS, a severe infectious disease. Since the identification of AIDS, the rapid spread of AIDS has positioned it as a principal infectious disease posing a significant threat to global health. Beyond its health implications, AIDS also engenders a spectrum of moral and ethical dilemmas. Consequently, investigating the pathogenesis, transmission dynamics, and strategies for the prevention and control of AIDS has emerged as a critical and pressing endeavor.

	Research indicates that HIV spreads within target cells through two primary mechanisms: virus-to-cell and cell-to-cell transmissions \cite{kimata1999emerging}. In recent years, mathematical modeling has become a pivotal tool for examining the dynamics of disease transmission between hosts and specifically, the intricacies of HIV infection within a host. An increasing number of scholars have focused their research on the latter—developing viral infection kinetic models that elucidate the interactions between CD4$^{+}$T cells and HIV. The foundational models for HIV-1, which outline the mechanics of viral infection disease, were introduced by \cite{nowak1997anti,nowak1996hiv}. These models establish the relationships among CD4$^{+}$T cells, infected CD4$^{+}$T cells, and the virus. Building upon this framework, \cite{wang2012global} and \cite{murase2005stability} each proposed a viral infection kinetic model that incorporates humoral immunity. A common assumption in these models is that all biological processes triggered by the virus's entry into the body occur instantaneously—an impractical supposition given the inherent intracellular time delays. To address this, numerous scholars have integrated time delays into the model, aiming to accurately reflect the impact of these delays on cell infection within the host \cite{herz1996viral,wang2007complex,shi2010dynamical,hattaf2013stability,chen2023complex,jiang2022global,lin2017threshold,guo2020analysis}.

	Recent studies into the mechanisms of CD4$^{+}$T cell death \cite{doitsh2014cell,wang2018caspase} have revealed that the secretion of inflammatory factors from dying cells attracts a significant influx of uninfected cells to the site of inflammation. This process leads to increased cell infection and death. However, earlier models \cite{herz1996viral,wang2007complex,shi2010dynamical,hattaf2013stability,chen2023complex,jiang2022global,lin2017threshold,guo2020analysis} did not consider the role of inflammatory factors. In response, Zhang et al. \cite{zhang2023dynamic} incorporated inflammatory factors into their model, which is presented as follows:

	\begin{equation}\label{eq1}
		\left\{
		\begin{aligned}
			\frac{dx(t)}{dt}=&\lambda-\beta_{1}x(t)v(t)-\beta_{2}x(t)c(t)-d_{1}x(t),\\
			\frac{dy(t)}{dt}=&e^{-m\tau_{1}}\beta_{1}x(t-\tau_{1})v(t-\tau_{1})+e^{-m\tau_{1}}\beta_{2}x(t-\tau_{1})c(t-\tau_{1})-(\alpha_{1}+d_{2})y(t)\\
			&-py(t)z(t),\\
			\frac{dc(t)}{dt}=&\alpha_{2}y(t)-d_{3}c(t),\\
			\frac{dv(t)}{dt}=&ke^{-n\tau_{2}}y(t-\tau_{2})-d_{4}v(t),\\
			\frac{dz(t)}{dt}=&cy(t-\tau_{3})z(t-\tau_{3})-d_{5}z(t).\\
		\end{aligned}
		\right.
	\end{equation}
	The meanings of the system's variables and parameters are referenced in the table below:
	
	{\scriptsize\begin{longtable}{|c|c|c|}
			\caption{Description of Dynamic Variables and Parameters} \label{tab:longtable} \\
			\hline \textbf{Symbol} & \textbf{Description} & \textbf{Type} \\ \hline 
			\endfirsthead
			
			\multicolumn{3}{c}%
			{{\bfseries \tablename\ \thetable{} -- continued from previous page}} \\
			\hline \textbf{Symbol} & \textbf{Description} & \textbf{Type} \\ \hline 
			\endhead
			
			\hline \multicolumn{3}{|r|}{{Continued on next page}} \\ \hline
			\endfoot
			
			\hline \hline
			\endlastfoot
			\small
			$x(t)$ & Concentration of uninfected CD4$^{+}$T cells & Dynamic Variable \\ \hline
			$y(t)$ & Concentration of infected CD4$^{+}$T cells & Dynamic Variable \\ \hline
			$c(t)$ & Concentration of inflammatory cytokines & Dynamic Variable \\ \hline
			$v(t)$ & Concentration of free viruses & Dynamic Variable \\ \hline
			$z(t)$ & Concentration of CTL immune response cells & Dynamic Variable \\ \hline
			$\lambda$ & Proliferation rate of uninfected CD4$^{+}$T cells & Parameter \\ \hline
			$\alpha_1$ & Mortality rate of infected CD4$^{+}$T cells due to pyroptosis & Parameter \\ \hline
			$\alpha_2$ & Proliferation rate of inflammatory cytokines & Parameter \\ \hline
			$\beta_1$ & Rate of CD4$^{+}$T cells infection by viruses & Parameter \\ \hline
			$\beta_2$ & Rate of CD4$^{+}$T cells infection by inflammatory cytokines & Parameter \\ \hline
			$k$ & Proliferation rate of viruses & Parameter \\ \hline
			$c$ & Proliferation rate of CTL immune cells & Parameter \\ \hline
			$p$ & Rate at which CTL immune cells kill infected CD4$^{+}$T cells & Parameter \\ \hline
			$\tau_1$ & Duration from virus entering the cell to production of new virions & Parameter \\ \hline
			$\tau_2$ & Time for a virus to replicate and produce a new virus & Parameter \\ \hline
			$\tau_3$ & Time from antigenic stimulation to the production of CTL immune cells & Parameter \\ \hline
			$d_1$ & Natural mortality rate of uninfected CD4$^{+}$T cells & Parameter \\ \hline
			$d_2$ & Natural mortality rate of infected CD4$^{+}$T cells & Parameter \\ \hline
			$d_3$ & Natural mortality rate of inflammatory cytokines & Parameter \\ \hline
			$d_4$ & Natural mortality rate of free viruses & Parameter \\ \hline
			$d_5$ & Natural mortality rate of CTL immune responsive cells & Parameter \\ \hline
	\end{longtable}}

	We observe that system (\ref{eq1}) assumes that the CTL immune response can be activated at a bilinear rate, which is not precise. To develop a more biologically meaningful mathematical model, many researchers have suggested replacing the bilinear incidence with a nonlinear rate. A saturated immune response function, $c \frac{y(t)z(t)}{h+z(t)}$, is employed in \cite{jiang2014complete,ren2021global} instead of the simpler $cy(t)z(t)$, where $h$ represents a saturation constant. Concurrently, distributed delays have been incorporated into the models by \cite{mittler1998influence,wang2016threshold,nelson2002mathematical,wang2023hiv,Su2015GlobalAO,Liang2014GlobalAO,YANG2015183}.

	Based on the preceding discussion, in this paper, we extend system (\ref{eq1}) by incorporating distributed delays \cite{ YANG2015183,wang2023hiv} and a saturated infection rate \cite{jiang2014complete}. The model under consideration is presented as follows:
	
	\begin{equation}\label{eq2}
		\left\{
		\begin{aligned}
			\frac{dx(t)}{dt}=&\lambda-\beta_{1}x(t)v(t)-\beta_{2}x(t)c(t)-d_{1}x(t),\\
			\frac{dy(t)}{dt}=&\int_{0}^{\infty}f_{1}(s)e^{-m_{1}s}\beta_{1}x(t-s)v(t-s)ds+\int_{0}^{\infty}f_{1}(s)e^{-m_{1}s}\beta_{2}x(t-s)c(t-s)ds\\
			&-(\alpha_{1}+d_{2})y(t)-py(t)z(t),\\
			\frac{dc(t)}{dt}=&\alpha_{2}y(t)-d_{3}c(t),\\
			\frac{dv(t)}{dt}=&k\int_{0}^{\infty}f_{2}(s)e^{-m_{2}s}y(t-s)ds-d_{4}v(t),\\
			\frac{dz(t)}{dt}=&c\frac{y(t)z(t)}{h+z(t)}-d_{5}z(t),\\
		\end{aligned}
		\right.
	\end{equation}
	where $f_1(s)$ and $f_2(s)$ represent probability distributions, and $s$ is defined as a random variable;  other variables and parameters retain the meanings assigned to them in Table 1.   We postulate that if a virus or an infected cell makes contact with an uninfected CD4$^{+}$T cell at time $t-s$, the cell will become infected by time $t$. The expression $e^{-m_1 s}$ quantifies the survival rate of the cell throughout this delay period. Additionally, once infected at time $t-s$, the cell starts to produce new infectious viruses by time $t$, with $e^{-m_2 s}$ indicating the survival rate of the infected cell during this intervening period.

	The probability distribution function $f_{i}(s)$ is referred to as the delay kernel, and it satisfies the following properties:
	\begin{equation}
		f_{i}(s) \geq 0, \quad \int_{0}^{\infty} f_{i}(s) \, ds = 1, \quad i=1,2.
	\end{equation}

	Define the Banach space of fading memory type as follows:
	\[
	C_\alpha = \left\{ \phi \in C((-\infty, 0], \mathbb{R}) : \phi(\theta) e^{\alpha \theta} \text{ is uniformly continuous on } (-\infty, 0] \text{ and } \|\phi\| < \infty \right\},
	\]
	where \(0<\alpha < \frac{\min\{m_1, \,m_2\}}{2}\),  and the norm \(\|\phi\|\) is defined by:
	\[
	\|\phi\| = \sup_{\theta \leq 0} \left| \phi(\theta) e^{\alpha \theta} \right|.
	\]
	Additionally, define the positive subset of this space as:
	\[
	C_\alpha^+ = \left\{ \phi \in C_\alpha : \phi(\theta) \geq 0 \text{ for all } \theta \in (-\infty, 0] \right\}.
	\]

	We assume that the initial conditions for system (\ref{eq2}) are defined as:
	\begin{equation}\label{inc}
		\begin{aligned}
			&x(\theta) = \varphi_1(\theta), \quad y(\theta) = \varphi_2(\theta), \quad c(\theta) = \varphi_3(\theta), \quad v(\theta) = \varphi_4(\theta), \quad z(\theta) = \varphi_5(\theta), \\
			&\varphi_{i}\in C_\alpha^{+},  
			\quad i = 1,2,3,4,5, \quad \theta \in (-\infty, 0].
		\end{aligned}
	\end{equation}
	Based on the fundamental principles of functional differential equations outlined in reference \cite{hale2013introduction,MR1218880}, it can be demonstrated that model (\ref{eq2}) possesses a unique solution which conforms to the initial condition (\ref{inc}).

	The paper is organized as follows. Section 2 involves calculating the basic reproduction numbers and equilibrium points. Concurrently, we introduce a convex cone to prove the preservation of non-negativity of solutions with non-negative initial conditions, addressing a natural question. We also demonstrate the boundness of solutions under the same conditions.  In Section 3, by defining Lyapunov functions and employing LaSalle's invariance principle, we analyze the global stability of three equilibria: the disease-free equilibrium, the immunity-inactivated equilibrium, and the immunity-activated equilibrium. Section 4 deals the substitution of the probability distribution function with the Dirac delta function to simplify the model, with verification of the results through numerical simulations. The paper concludes in the final section.

	\section{Preliminary results of solution}
	In this section, we will calculate the reproduction numbers and subsequently discuss the positivity and boundedness of system (\ref{eq2}).
	
	Clearly, system (\ref{eq2}) has a disease-free equilibrium $E_0$, where
	\begin{equation}
		E_0 = \left(x_0, 0, 0, 0, 0\right) = \left(\frac{\lambda}{d_1}, 0, 0, 0, 0\right).
	\end{equation}
	
	Let $A_1 = \int_0^\infty f_1(s) e^{-m_1 s} \, ds$ and $A_2 = \int_0^\infty f_2(s) e^{-m_2 s} \, ds$. Following the methodology of \cite{van2002reproduction,zhang2023dynamic}, we define the basic reproduction number $\mathcal{R}_0$ as follows:
	\begin{equation}
		\mathcal{R}_0 = \frac{\beta_1 A_1 k A_2 d_3 x_0 + \beta_2 A_1 \alpha_2 x_0 d_4}{d_3 d_4 (\alpha_1 + d_2)}.
	\end{equation}
	
	It is straightforward to demonstrate that if $\mathcal{R}_0 > 1$, system (\ref{eq2}) exhibits an immunity-inactivated equilibrium $E_1(x_1, y_1, c_1, v_1, 0)$, where
	\begin{equation}
		x_1 = \frac{x_0}{\mathcal{R}_0}, \quad y_1 = \frac{d_1 d_3 d_4 (\mathcal{R}_0 - 1)}{\beta_1 k A_2 d_3 + \beta_2 \alpha_2 d_4}, \quad c_1 = \frac{\alpha_2}{d_3} y_1, \quad v_1 = \frac{k A_2}{d_4} y_1.
	\end{equation}
	
	We can obtain the CTL immune reproduction number $\mathcal{R}_1$ using a similar approach as for $\mathcal{R}_0$ in \cite{zhang2023dynamic}, where
	\begin{equation}
		\mathcal{R}_1 = \frac{c d_1 d_3 d_4 (\mathcal{R}_0 - 1)}{h d_5 (\beta_1 k A_2 d_3 + \beta_2 \alpha_2 d_4)}.
	\end{equation}
	Define
	\begin{equation}
		\begin{aligned}
			\Delta &= [d_5 (\alpha_1 + d_2 + ph) (\beta_1 k A_2 d_3 + \beta_2 \alpha_2 d_4) + d_1 d_3 d_4 pc]^2 \\
			&\quad + 4p (\beta_1 k A_2 d_3 + \beta_2 \alpha_2 d_4)^2 h (\alpha_1 + d_2) d_5^2 (\mathcal{R}_1 - 1).
		\end{aligned}
	\end{equation}
	If $\mathcal{R}_1 > 1$, there exists an immunity-activated equilibrium $E_2(x_2, y_2, c_2, v_2, z_2)$, where
	\begin{equation}
		\begin{aligned}
			x_{2}&=\frac{(\alpha_{1}+d_{2}+pz_{2})d_{3}d_{4}}{A_{1}(\beta_{1}kA_{2}d_{3}+\beta_{2}\alpha_{2}d_{4})},\enspace y_{2}=\frac{d_{5}(h+z_{2})}{c},\enspace c_{2}=\frac{\alpha_{2}d_{5}(h+z_{2})}{cd_{3}},\enspace v_{2}=\frac{kA_{2}d_{5}(h+z_{2})}{c d_{4}}, \\
			z_{2}&=\frac{-[d_{5}(\alpha_{1}+d_{2}+ph)(\beta_{1}kA_{2}d_{3}+\beta_{2}\alpha_{2}d_{4})+d_{1}d_{3}d_{4}pc]+\sqrt{\Delta}}{2pd_{5}(\beta_{1}kA_{2}d_{3}+\beta_{2}\alpha_{2}d_{4})}.
		\end{aligned}
	\end{equation}

	Using the same method as \cite{yang2023global}, we can prove the following theorem.
	\begin{theorem}\label{tho1}
		All solutions of system(\ref{eq2}) with positive initial conditions always stay positive.
	\end{theorem}

	\begin{proof} 
		For all $t \geq 0$, define $m(t) = \min \left\{x(t), y(t), c(t), v(t),  z(t)\right\}$. Given that the initial values are positive, it follows that $m(0) > 0$. To establish that $m(t)$ is positive  for all $t \geq 0$, suppose contrary to our claim, the system does not maintain positivity. Consequently, there must exist a time $t_1 > 0$ such that $m(t) > 0$ for $0 \leq t < t_1$ and $m(t) = 0$ at $t = t_1$.  By analyzing the behavior of $m(t_1)$, we can derive the following five cases:

		(1) If $m(t_1) = x(t_1) = 0$, from the first equation of system (\ref{eq2}), we obtain
		\begin{equation}
			\begin{aligned}
				\frac{dx(t)}{dt} &= \lambda - d_1 x(t) - \beta_1 x(t) v(t) - \beta_2 x(t) c(t) \\
				&\geq -d_1 x(t) - \max\{\beta_1 v(t)\} x(t) - \max\{\beta_2 c(t)\} x(t) \\
				&= -b_1 x(t),
			\end{aligned}
		\end{equation}
		for $t \in [0, t_1]$, where $b_1 = d_1 + \max_{t \in [0, t_1]} \{\beta_1 v(t)\} + \max_{t \in [0, t_1]} \{\beta_2 c(t)\}$. Consequently, $x(t_1) \geq x(0) e^{-b_1 t_1} > 0$, which contradicts the fact that $x(t_1) = 0$.

		(2)		If $m(t_1) = y(t_1) = 0$, from the second equation of system (\ref{eq2}), we obtain
		\begin{equation}
			\begin{aligned}
				\frac{dy(t)}{dt} &= \int_{0}^{\infty} f_1(s) e^{-m_1 s} \beta_1 x(t-s) v(t-s) \, ds + \int_{0}^{\infty} f_1(s) e^{-m_1 s} \beta_2 x(t-s) c(t-s) \, ds \\
				&\quad - (\alpha_1 + d_2) y(t) - p y(t) z(t) \\
				&\geq -(\alpha_1 + d_2) y(t) - p \max_{t \in [0, t_1]} \{z(t)\} y(t) \\
				&= -b_2 y(t),
			\end{aligned}
		\end{equation}
		for $t \in [0, t_1]$, where $b_2 = \alpha_1 + d_2 + p \max_{t \in [0, t_1]} \{z(t)\}$. Consequently, $y(t_1) \geq y(0) e^{-b_2 t_1} > 0$, which contradicts the fact that $y(t_1) = 0$.

		(3) If $m(t_1) = c(t_1) = 0$, according to  the third equation of system (\ref{eq2}), we derive
		\begin{equation}
			\frac{dc(t)}{dt} = \alpha_2 y(t) - d_3 c(t) \geq -d_3 c(t),
		\end{equation}
		for $t \in [0, t_1]$. Consequently, $c(t_1) \geq c(0) e^{-d_3 t_1} > 0$, which contradicts the fact that $c(t_1) = 0$.

		(4)	If $m(t_1) = v(t_1) = 0$, according to the fourth equation of system(\ref{eq2}), we have
		\begin{equation}
			\frac{dv(t)}{dt}=k\int_{0}^{\infty}f_{2}(s)e^{-m_{2}s}y(t-s)ds-d_{4}v(t)\ge-d_{4}v(t),
		\end{equation}
		for $t \in [0, t_1]$. It follows that  $v(t_1)\ge v(0)e^{-d_{4}t_1}>0$, which contradicts to the fact $v(t_{1})=0$.
		
		(5) If $m(t_1) = z(t_1) = 0$, 	by the fifth equation of system(\ref{eq2}), we have 
			\begin{equation}
			\frac{dz(t)}{dt}=c\frac{y(t)z(t)}{h+z(t)}-d_{5}z(t)\ge-d_{5}z(t).
		\end{equation}
		for $t \in [0, t_1]$.	We conclude that  $z(t_1)\ge z(0)e^{-d_{5}t_1}>0$, which contradicts  the fact $z(t_{1})=0$.
	\end{proof}

A pertinent question in this field is whether non-negative initial values guarantee non-negative solutions, a problem that poses considerable challenges. In \cite[Lemma 2 and Lemma 3]{MR1403462}, Yang et al. introduced a criterion to analyze two specific systems, for an analogous result, see \cite[Theorem 2.1]{MR1319817}. Recent studies \cite{yang2023global, MR4034240} have focused on scenarios characterized by positive initial values.  We adopt a novel approach to examine the impact of non-negative initial values on solutions. Specifically, we introduce the concept of a convex cone ${K_5}$ and demonstrate that ${K_5}$ remains invariant under our system. 

 We define a convex cone by 
\[ K_5 = \{(x, y, c, v, z) \in \mathbb{R}^5 : x \geq 0, y \geq 0, c \geq 0, v \geq 0, z \geq 0\} \]
and the interior of \( K_5 \) by 
\[ \mathring{K}_5 = \{(x, y, c, v, z) \in \mathbb{R}^5 : x > 0, y > 0, c > 0, v > 0, z > 0\}. \]
It is clear that \( K_5\) is a convex subset of \(\mathbb{R}^5\). Indeed, Theorem \ref{tho1} essentially states that if the initial conditions resides inside \( \mathring{K}_5  \), then it will remain inside \( \mathring{K}_5  \).  

Our next theorem addresses the boundary of \( K_5 \). Specifically, it indicates that if the solution reaches the boundary \(\partial K_5\), it will either push the solution back inside the cone or maintain its position on the boundary. We define vector fields
\[ X= (x, y, c, v, z) \]
and 
\[ \dot{X}= \left(\frac{dx}{dt}, \frac{dy}{dt}, \frac{dc}{dt}, \frac{dv}{dt}, \frac{dz}{dt}\right).\]

Assume we have a vector $\dot{X}=\left(\frac{dx}{dt}, \frac{dy}{dt}, \frac{dc}{dt}, \frac{dv}{dt}, \frac{dz}{dt}\right)$ at a point $X=(x, y, c, v, z)$ on the boundary of $K_5$. The conditions for the vector $\dot{X}$ to point inside of $K_5$ or remain on the boundary are:

1. For any coordinate of $X$ that is equal to zero, the corresponding component of $\dot{X}$ should be non-negative.

2. For non-zero coordinate of $X$, the corresponding components of $\dot{X}$ can be any real number.

Our next result asserts that \( \dot{X}=|_{\partial K_5} \) either points inside \( K_5 \) or remains tangent to $\partial K_5$.

\begin{figure}
	\centering
	\begin{tikzpicture}
		\begin{axis}[
			view={120}{30}, 
			axis lines=center, 
			xlabel={$y,c,v$}, ylabel={$z$}, zlabel={$x$}, 
			xmax=5, ymax=5, zmax=5, 
			xmin=0, ymin=0, zmin=0, 
			ticks=none, 
			axis line style={->}, 
			xlabel style={anchor=south}, 
			ylabel style={anchor=south}, 
			zlabel style={anchor=south west}, 
			every axis plot post/.append style={
				mark options={solid}
			}
			]
			\node[label={180:{O}},circle,fill,inner sep=1pt] at (axis cs:0,0,0) {};
		\end{axis}
		\node at ([xshift=2.0cm, yshift=0cm]current axis.north west) {$\mathbb{R}^5$};
	\end{tikzpicture}
	\caption{a convex cone $K_5$}
\end{figure}
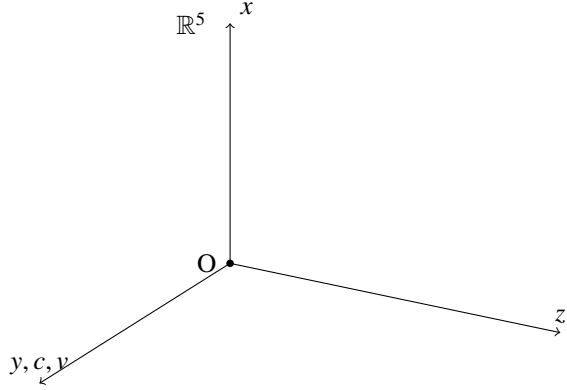

\begin{theorem}\label{tho2}
Suppose that at \( t_0 \geq 0 \), \( X(t_0) \in \partial K_5 \) and \( X(t) \in K_5 \) for all \( t < t_0 \). Then the vector \(\dot{X}(t_0)\) either points inside  of \( K_5 \) or remains tangent to the boundary \( \partial K_5 \). Moreover, \( X(t) \) remains within \( K_5 \) for all \( t \geq t_0 \).

\end{theorem}  
\begin{proof}

Consider the function
\[
f(t, X_t) = (f_1(t, X_t), f_2(t, X_t), f_3(t, X_t), f_4(t, X_t), f_5(t, X_t)),
\]
where \(X_t(\theta) = X(t + \theta)\) for \(\theta \in (-\infty, 0]\), and \(f_i(t, X_t)\) denotes the right-hand side of the i-th equation in system (\ref{eq2}). We express system (\ref{eq2}) as
\begin{equation}
\label{mgf}
\dot{X}(t) = f(t, X_t).
\end{equation}
If \(X(t_0) \in \partial K_5\), then the coordinates of \(X(t_0)\) include at least one zero component. We analyze several cases:

\begin{itemize}
    \item \textbf{Case One:} If \(x(t_0) = 0\), then from the first equation in system (\ref{eq2}), it follows that \(f_1(t_0, X_{t_0}) > 0\).
    \item \textbf{Case Two:} Assuming \(X(t) \in K_5\) for all \(t < t_0\), the integral
    \[
    \int_{0}^{\infty} f_1(s) e^{-m_1 s} (\beta_1 x(t_0 - s) v(t_0 - s) + \beta_2 x(t_0 - s) c(t_0 - s)) \, ds \geq 0.
    \]
    If \(y(t_0) = 0\), then from the second equation in system (\ref{eq2}), we have \(f_2(t_0, X_{t_0}) \geq 0\).
    \item \textbf{Case Three:} If \(c(t_0) = 0\), then by the third equation and under the condition \(X(t) \in K_5\) for all \(t < t_0\), it follows that \(f_3(t_0, X_{t_0}) \geq 0\).
    \item \textbf{Case Four:} If \(v(t_0) = 0\), then by the fourth equation, under the same condition,
    \[
    f_4(t_0, X_{t_0}) = k \int_{0}^{\infty} f_2(s) e^{-m_2 s} y(t_0 - s) ds \geq 0.
    \]
    \item \textbf{Case Five:} If \(z(t_0) = 0\), then by the fifth equation in system (\ref{eq2}), we have \(f_5(t_0, X_{t_0}) = 0\).
\end{itemize}

Thus, \(\dot{X}(t_0)\) always points inside of \(K_5\) or remains tangent to \(\partial K_5\). To substantiate our last conclusion, we employ the methods described in \cite[Lemma 2]{MR1403462} and \cite[Theorem 1]{MR3681108}. Consider the modified equation
\begin{equation}
\label{mga}
\dot{X}(t) = f(t, X_t) + \frac{1}{N} \hat{e},
\end{equation}
where \(N\) is any positive integer and \(\hat{e} = (1, 1, 1, 1, 1)\). Let \(m(t) = \min \{x(t), y(t), c(t), v(t), z(t)\}\). If \(m(t)\) is not non-negative, there exists a first time \(t_1 >t_0\)  such that \(m(t_1) = 0\) and \(m(t) > 0\) for \(t \in (t_0, t_1)\). Without loss of generality, assuming that \(m(t_1) = y(t_1)\), we observe that \(X(t) \geq 0\) for \(t \leq t_1\) and \(f_2(t_1, X_{t_1}) + \frac{1}{N} > 0\), this contradicts \(\frac{dy}{dt}\big|_{t = t_1} \leq 0\). Consequently, the solution to (\ref{mga}) remains in \(K_5\) for \(t \geq t_0\). Letting \(N \rightarrow \infty\), we conclude that the solution to (\ref{eq2}) stays in \(K_5\) for \(t \geq t_0\).

\end{proof}

Combining Theorems \ref{tho1}  and \ref{tho2}, we conclude that all solutions of system(\ref{eq2}) with   non-negative  initial conditions are always non-negative.  
Furthermore, the proof of Theorems \ref{tho1} demonstrates that a solution, provided non-negative initial conditions, remains within $\mathring{K}_5$ for all $t > t_0$ once it is in $\mathring{K}_5$ at $t_0$.

In fact, the  method described in Theorem \ref{tho2} can be easily generalized to a large class of systems.  We define a convex cone in $\mathbb{R}^n$ as follows:
\[ K_n = \left\{(x_1, x_2, \cdots, x_n) \in \mathbb{R}^n : x_1 \geq 0, x_2 \geq 0, \cdots, x_n \geq 0\right\}, \] and its boundary is denoted by $\partial K_n$. 
Consider the differential equation with delay:
\begin{equation}\label{mge}
\dot{X}(t) = f(t, X_t), 
\end{equation}
where the delay can be either finite or infinite. Here, $X(t) = \left(x_1(t), x_2(t), \cdots, x_n(t)\right)$, and $X_t(\theta) = X(t + \theta)$ for $\theta$ in the time delay interval. The function $f(t, X_t) = \left(f_1(t, X_t), f_2(t, X_t), \cdots, f_n(t, X_t)\right)$ is assumed to be Lipschitz continuous in its second argument on each compact subset of the domain of $f$, ensuring that the system (\ref{mge}) admits a unique solution for each initial condition in an appropriate space.

Using the same method in Theorem {\ref{tho2}}, we can prove the following result.

\begin{theorem}
For any \( t_0 \geq 0 \), if \( X(t_0) \in \partial K_n \),  \( X(t) \) stays within \( K_n \) for all \( t < t_0 \), and  the vector \( f(t_0, X_{t_0})\) either points inside of \( K_n \) or remains tangent to the boundary \( \partial K_n \). 
Then \( X(t) \) will continue to stay within \( K_n \) for all \( t \geq t_0 \).
\end{theorem}

Next, we prove the boundedness  of solutions to system (\ref{eq2}). 
	\begin{theorem}
		All solutions of system (\ref{eq2}) with   non-negative  initial conditions are always bounded. 
	\end{theorem}
	\begin{proof}
		Let $(x(t), y(t), c(t), v(t), z(t))$ be a solution of system (\ref{eq2}) with non-negative initial conditions. 
		
		We define
		\begin{equation}
			B(t) = \int_{0}^{\infty} e^{-m_{1}s} f_{1}(s) x(t-s) \, ds + y(t).
		\end{equation}
		
		Let $r = \min\{d_{1}, (\alpha_{1} + d_{2})\}>0$. We can calculate the derivative of $B(t)$ with respect to $t$ as follows:
		\begin{equation}
			\begin{aligned}
				\frac{dB(t)}{dt} &= \int_{0}^{\infty} f_{1}(s) e^{-m_{1}s} (\lambda - d_{1} x(t-s) - \beta_{1} x(t-s) v(t-s) - \beta_{2} x(t-s) c(t-s)) \, ds \\
				&\quad + \int_{0}^{\infty} f_{1}(s) e^{-m_{1}s} \beta_{1} x(t-s) v(t-s) \, ds + \int_{0}^{\infty} f_{1}(s) e^{-m_{1}s} \beta_{2} x(t-s) c(t-s) \, ds \\
				&\quad - (\alpha_{1} + d_{2}) y(t) - p y(t) z(t) \\
				&= \lambda A_{1} - d_{1} \int_{0}^{\infty} f_{1}(s) e^{-m_{1}s} x(t-s) \, ds - (\alpha_{1} + d_{2}) y(t) - p y(t) z(t) \\
				&\leq \lambda A_{1} - \left[d_{1} \int_{0}^{\infty} f_{1}(s) e^{-m_{1}s} x(t-s) \, ds + (\alpha_{1} + d_{2}) y(t)\right] \\
				&\leq \lambda A_{1} - r B(t).
			\end{aligned}
		\end{equation}
		This implies that
		\begin{equation} \label{bes}
		B(t) \leq \frac{\lambda A_{1}}{r}+\frac{rB_0-\lambda A_1}{re^{rt}}. 
		\end{equation}
		Hence it follows that,  
        \begin{equation}
		y(t) \leq \frac{\lambda A_{1}}{r}+\frac{rB_0-\lambda A_1}{re^{rt}}\leq C_1,\end{equation} here $C_1=\frac{\lambda A_{1}}{r}+\frac{|rB_0-\lambda A_1|}{r}$. 
		From the first, the third, fourth, and fifth equations of model (\ref{eq2}), we have
		\begin{equation}
			\begin{aligned}
				\frac{dx(t)}{dt} &= \lambda - d_{1} x(t) - \beta_{1} x(t) v(t) - \beta_{2} x(t) c(t) \leq \lambda - d_{1} x(t), \\
				\frac{dc(t)}{dt} &= \alpha_{2} y(t) - d_{3} c(t) \leq C_2 - d_{3} c(t), \\
				\frac{dv(t)}{dt} &= k \int_{0}^{\infty} f_{2}(s) e^{-m_{2}s} y(t-s) \, ds - d_{4} v(t) \leq C_3- d_{4} v(t), \\
				\frac{dz(t)}{dt} &= c \frac{y(t) z(t)}{h + z(t)} - d_{5} z(t) \leq c y(t) - d_{5} z(t) \leq C_4 - d_{5} z(t),
			\end{aligned}
		\end{equation}  where  $C_2$, $C_3$, and $C_4$ are positive constants. 
	By similar estimates as in  (\ref{bes}),  we conclude that 
        $x(t), c(t), v(t), z(t)$ are all bounded.

	\end{proof}
	\section{Globally asymptotically stable}
	 An equilibrium point of the system (\ref{eq2}) is called globally asymptotically stable if it is stable and, regardless of the initial conditions anywhere in the state space (not just within a specific neighborhood), the system's solution converges to the equilibrium point as time approaches infinity. In this section, we discuss the global stability of $E_{1},E_{2}$ and $E_{3}$ by defining Lyapunov functions and LaSalle's invariance principle. In the beginning, we define $g(s)=s-1-\ln{s}$.
	
	\begin{theorem}\label{theorem1}
		If $\mathcal{R}_{0}<1$, $E_{0}$ is globally asymptotically stable.
	\end{theorem}
	\begin{proof}
		We define 
		\begin{equation}
			\begin{aligned}
				N_{01} &= \int_{0}^{\infty} \int_{t-s}^{t} f_{1}(s) e^{-m_{1}s} \beta_{1} x(\eta) v(\eta) \, d\eta \, ds, \\
				N_{02} &= \int_{0}^{\infty} \int_{t-s}^{t} f_{1}(s) e^{-m_{1}s} \beta_{2} x(\eta) c(\eta) \, d\eta \, ds, \\
				N_{03} &= k \int_{0}^{\infty} \int_{t-s}^{t} f_{2}(s) e^{-m_{2}s} y(\eta) \, d\eta \, ds.
			\end{aligned}
		\end{equation}
		Differentiating $N_{0i}$ (for $i=1, 2, 3$) with respect to $t$, we obtain
		\begin{equation}
			\begin{aligned}
				\frac{dN_{01}}{dt} &= A_{1} \beta_{1} x(t) v(t) - \int_{0}^{\infty} f_{1}(s) e^{-m_{1}s} \beta_{1} x(t-s) v(t-s) \, ds, \\
				\frac{dN_{02}}{dt} &= A_{1} \beta_{2} x(t) c(t) - \int_{0}^{\infty} f_{1}(s) e^{-m_{1}s} \beta_{2} x(t-s) c(t-s) \, ds, \\
				\frac{dN_{03}}{dt} &= k A_{2} y(t) - k \int_{0}^{\infty} f_{2}(s) e^{-m_{2}s} y(t-s) \, ds.
			\end{aligned}
		\end{equation}
		We define a Lyapunov function as follows:
		\begin{equation}
			\begin{aligned}
				L_{0}(t) &= A_{1} x_{0} g\left(\frac{x(t)}{x_{0}}\right) + y(t) + N_{01} + N_{02} + \frac{\beta_{2} A_{1} x_{0}}{d_{3} \mathcal{R}_{0}} c(t) + \frac{\beta_{1} A_{1} x_{0}}{d_{4} \mathcal{R}_{0}} v(t) \\
				&\quad + \frac{ph}{c} z(t) + \frac{\beta_{1} A_{1} x_{0}}{d_{4} \mathcal{R}_{0}} N_{03}.
			\end{aligned}
		\end{equation}
		Calculating the derivative of $L_{0}(t)$, we obtain
			\begin{equation}
			\begin{aligned}
				\frac{dL_{0}(t)}{dt} &= A_{1} d_{1} x_{0} \left(2 - \frac{x_{0}}{x(t)} - \frac{x(t)}{x_{0}}\right) + \frac{A_{1} \beta_{1} x_{0}}{\mathcal{R}_{0}} (\mathcal{R}_{0} - 1) v(t) + \frac{A_{1} \beta_{2} x_{0}}{\mathcal{R}_{0}} (\mathcal{R}_{0} - 1) c(t) \\
				&\quad + p y(t) z(t) \left(\frac{h}{h+z(t)} - 1\right) - \frac{ph d_{5}}{c} z(t).
			\end{aligned}
		\end{equation}
		Note that 
		\begin{equation}
			\begin{aligned}
				&2 - \frac{x_{0}}{x(t)} - \frac{x(t)}{x_{0}} \leq 0, \quad x(t) \geq 0, \\
				&\frac{h}{h+z(t)} - 1 \leq 0, \quad z(t) \geq 0.
			\end{aligned}
		\end{equation}
		If $\mathcal{R}_{0} < 1$, we have $\frac{dL_{0}(t)}{dt} \leq 0$. Moreover, $\frac{dL_{0}(t)}{dt} = 0$ if and only if $x(t) = x_{0}$ and $$y(t) = c(t) = v(t) = z(t) = 0.$$ $\mathcal{M}_{0} = \{E_{0}\}$ is the largest invariant subset of $\{(x(t), y(t), c(t), v(t), z(t)) \mid \frac{dL_{0}(t)}{dt} = 0\}$. From LaSalle's invariance principle, we conclude  that $E_{0}$ is globally asymptotically stable provided  $\mathcal{R}_{0} < 1$.
	\end{proof}

	\begin{theorem}\label{theorem2}
		If $\mathcal{R}_{1}<1<\mathcal{R}_{0}$, then $E_{1}$ is globally asymptotically stable.
	\end{theorem}
	\begin{proof}
		We define
		\begin{equation}
			\begin{aligned}
				N_{11}&=\beta_{1}x_{1}v_{1}\int_{0}^{\infty}\int_{t-s}^{t}f_{1}(s)e^{-m_{1}s}\left(\frac{x(\eta)v(\eta)}{x_{1}v_{1}}-1-\ln{\frac{x(\eta)v(\eta)}{x_{1}v_{1}}}\right)d\eta ds,\\
				N_{12}&=\beta_{2}x_{1}c_{1}\int_{0}^{\infty}\int_{t-s}^{t}f_{1}(s)e^{-m_{1}s}\left(\frac{x(\eta)c(\eta)}{x_{1}c_{1}}-1-\ln{\frac{x(\eta)c(\eta)}{x_{1}c_{1}}}\right)d\eta ds,\\
				N_{13}&=\frac{k\beta_{1}x_{1}y_{1}}{d_{4}}\int_{0}^{\infty}\int_{t-s}^{t}f_{2}(s)e^{-m_{2}s}\left(\frac{y(\eta)}{y_{1}}-1-\ln{\frac{y(\eta)}{y_{1}}}\right)d\eta ds.
			\end{aligned}
		\end{equation}
		Differentiating $N_{1i} (i=1, 2, 3)$ with respect to $t$, we obtain that
		\begin{equation}
			\begin{aligned}
				\frac{dN_{11}}{dt}&=A_{1}\beta_{1}x(t)v(t)-\int_{0}^{\infty}f_{1}(s)e^{-m_{1}s}\beta_{1}x(t-s)v(t-s)ds\\
				&\enspace+\beta_{1}x_{1}v_{1}\int_{0}^{\infty}f_{1}(s)e^{-m_{1}s}\ln{\frac{x(t-s)v(t-s)}{x(t)v(t)}}ds,\\
				\frac{dN_{12}}{dt}&=A_{1}\beta_{2}x(t)c(t)-\int_{0}^{\infty}f_{1}(s)e^{-m_{1}s}\beta_{2}x(t-s)c(t-s)ds\\
				&\enspace+\beta_{2}x_{1}c_{1}\int_{0}^{\infty}f_{1}(s)e^{-m_{1}s}\ln{\frac{x(t-s)c(t-s)}{x(t)c(t)}}ds,\\
				\frac{dN_{13}}{dt}&=\frac{k\beta_{1}x_{1}y_{1}}{d_{4}}\int_{0}^{\infty}f_{2}(s)e^{-m_{2}s}\left(\frac{y(t)}{y_{1}}-\frac{y(t-s)}{y_{1}}+\ln{\frac{y(t-s)}{y(t)}}\right)ds.
			\end{aligned}
		\end{equation}
		Define a Lyapunov function as follows:
		\begin{equation}
			\begin{aligned}
				L_{1}(t)&=x_{1}g\left(\frac{x(t)}{x_{1}}\right)+\frac{1}{A_{1}}y_{1}g\left(\frac{y(t)}{y_{1}}\right)+\frac{\beta_{2}x_{1}c_{1}}{d_{3}}g\left(\frac{c(t)}{c_{1}}\right)+\frac{\beta_{1}x_{1}v_{1}}{d_{4}}g\left(\frac{v(t)}{v_{1}}\right)+\frac{ph}{cA_{1}}z(t)\\
				&\enspace+\frac{1}{A_{1}}N_{11}+\frac{1}{A_{1}}N_{12}+N_{13}.
			\end{aligned}
		\end{equation}
		For $E_{1}$, it is easy to know that the following relationship is established:
		\begin{equation}
			\begin{aligned}
				&\lambda-d_{1}x_{1}-\beta_{1}x_{1}v_{1}-\beta_{2}x_{1}c_{1}=0,\\
				&\beta_{1}A_{1}x_{1}v_{1}+\beta_{2}A_{1}x_{1}c_{1}-(\alpha_{1}+d_{2})y_{1}=0.
			\end{aligned}
		\end{equation}
		Calculating the derivative of $L_{1}(t)$, we obtain
			\begin{equation}
			\begin{aligned}
				\frac{dL_{1}(t)}{dt}&=d_{1}x_{1}\left(2-\frac{x_{1}}{x(t)}-\frac{x(t)}{x_{1}}\right)-(\beta_{1}x_{1}v_{1}+\beta_{2}x_{1}c_{1})g\left(\frac{x_{1}}{x(t)}\right)\\
				&\enspace-\frac{\beta_{1}x_{1}v_{1}}{A_{1}}\int_{0}^{\infty}f_{1}(s)e^{-m_{1}s}g\left(\frac{x(t-s)v(t-s)y_{1}}{x_{1}v_{1}y(t)}\right)ds\\
				&\enspace-\frac{\beta_{2}x_{1}c_{1}}{A_{1}}\int_{0}^{\infty}f_{1}(s)e^{-m_{1}s}g\left(\frac{x(t-s)c(t-s)y_{1}}{x_{1}c_{1}y(t)}\right)ds\\
				&\enspace-\frac{\beta_{1}x_{1}v_{1}}{A_{2}}\int_{0}^{\infty}f_{2}(s)e^{-m_{2}s}g\left(\frac{y(t-s)v_{1}}{v(t)y_{1}}\right)ds\\
				&\enspace-\beta_{2}x_{1}c_{1}g\left(\frac{c_{1}y(t)}{c(t)y_{1}}\right)+\frac{p}{A_{1}}y(t)z(t)\left(\frac{h}{h+z(t)}-1\right)+\frac{phd_{5}}{A_{1}c}(\mathcal{R}_{1}-1)z(t).
			\end{aligned}
		\end{equation}
		Note that
		\begin{equation}
			\begin{aligned}
				&g(s)=s-1-\ln s\ge0,\enspace s>0,\\
				&2-\frac{x_{1}}{x(t)}-\frac{x(t)}{x_{1}}\le0,\enspace x(t)\ge0,\\
				&\frac{h}{h+z(t)}-1\le0,\enspace z(t)\ge0.
			\end{aligned}
		\end{equation}
		It follows that  $\frac{dL_{1}(t)}{dt}\le0$ and $\frac{dL_{1}(t)}{dt}=0$ occurs at $E_{1}$. Hence  $\mathcal{M}_{1}=\{E_{1}\}$ is the largest invariant subset of $\Big\{(x(t),y(t),c(t),v(t),z(t)) \mid \frac{dL_{0}(t)}{dt}=0\Big\}$. By  LaSalle's invariance principle, we  deduce that if $\mathcal{R}_{1}\le1\le\mathcal{R}_{0}$, then $E_{1}$ is globally asymptotically stable.
	\end{proof}
	\begin{theorem}\label{theorem3}
		If $\mathcal{R}_{0}$, $\mathcal{R}_{1}>1$, $E_{2}$ is globally asymptotically stable.
	\end{theorem}
	\begin{proof}
		We define
		\begin{equation}
			\begin{aligned}
				N_{21}&=\beta_{1}x_{2}v_{2}\int_{0}^{\infty}\int_{t-s}^{t}f_{1}(s)e^{-m_{1}s}\left(\frac{x(\eta)v(\eta)}{x_{2}v_{2}}-1-\ln{\frac{x(\eta)v(\eta)}{x_{2}v_{2}}}\right)d\eta ds,\\
				N_{22}&=\beta_{2}x_{2}c_{2}\int_{0}^{\infty}\int_{t-s}^{t}f_{1}(s)e^{-m_{1}s}\left(\frac{x(\eta)c(\eta)}{x_{2}c_{2}}-1-\ln{\frac{x(\eta)c(\eta)}{x_{2}c_{2}}}\right)d\eta ds,\\
				N_{23}&=\frac{k\beta_{1}x_{2}y_{2}}{d_{4}}\int_{0}^{\infty}\int_{t-s}^{t}f_{2}(s)e^{-m_{2}s}\left(\frac{y(\eta)}{y_{2}}-1-\ln{\frac{y(\eta)}{y_{2}}}\right)d\eta ds.
			\end{aligned}
		\end{equation}
		Differentiating $N_{2i} (i=1, 2, 3)$ with respect to $t$, we obtain that
		\begin{equation}
			\begin{aligned}
				\frac{dN_{21}}{dt}&=A_{1}\beta_{1}x(t)v(t)-\int_{0}^{\infty}e^{-m_{1}s}f_{1}(s)\beta_{1}x(t-s)v(t-s)ds\\
				&\enspace+\beta_{1}x_{2}v_{2}\int_{0}^{\infty}f_{1}(s)e^{-m_{1}s}\ln{\frac{x(t-s)v(t-s)}{x(t)v(t)}}ds,\\
				\frac{dN_{22}}{dt}&=A_{1}\beta_{2}x(t)c(t)-\int_{0}^{\infty}f_{1}(s)e^{-m_{1}s}\beta_{2}x(t-s)c(t-s)ds\\
				&\enspace+\beta_{2}x_{2}c_{2}\int_{0}^{\infty}f_{1}(s)e^{-m_{1}s}\ln{\frac{x(t-s)c(t-s)}{x(t)c(t)}}ds,\\
				\frac{dN_{23}}{dt}&=\frac{k\beta_{1}x_{2}y_{2}}{d_{4}}\int_{0}^{\infty}f_{2}(s)e^{-m_{2}s}\left(\frac{y(t)}{y_{2}}-\frac{y(t-s)}{y_{2}}+\ln{\frac{y(t-s)}{y(t)}}\right)ds.
			\end{aligned}
		\end{equation}
		Define a Lyapunov function as follows:
		\begin{equation}
			\begin{aligned}
				L_{2}(t)&=x_{2}g\left(\frac{x(t)}{x_{2}}\right)+\frac{1}{A_{1}}y_{2}g\left(\frac{y(t)}{y_{2}}\right)+\frac{\beta_{2}x_{2}c_{2}}{d_{3}}g\left(\frac{c(t)}{c_{2}}\right)+\frac{\beta_{1}x_{2}v_{2}}{d_{4}}g\left(\frac{v(t)}{v_{2}}\right)\\
				&\enspace+\frac{py_{2}}{d_{5}A_{1}}z_{2}g\left(\frac{z(t)}{z_{2}}\right)+\frac{1}{A_{1}}N_{21}+\frac{1}{A_{1}}N_{22}+N_{23}.
			\end{aligned}
		\end{equation}
		Note that 
		\begin{equation}
			\begin{aligned}
				&\lambda-d_{1}x_{2}-\beta_{1}x_{2}v_{2}-\beta_{2}x_{2}c_{2}=0,\\
				&\beta_{1}A_{1}x_{2}v_{2}+\beta_{2}A_{1}x_{2}c_{2}-(\alpha_{1}+d_{2})y_{2}-py_{2}z_{2}=0.
			\end{aligned}
		\end{equation}
		Calculating the derivative of $L_{2}(t)$, we obtain
		\begin{equation}
			\begin{aligned}
				\frac{dL_{2}(t)}{dt}&=d_{1}x_{2}\left(2-\frac{x_{2}}{x(t)}-\frac{x(t)}{x_{2}}\right)-(\beta_{1}x_{2}v_{2}+\beta_{2}x_{2}c_{2})g\left(\frac{x_{2}}{x(t)}\right)\\
				&\enspace-\frac{\beta_{1}x_{2}v_{2}}{A_{1}}\int_{0}^{\infty}f_{1}(s)e^{-m_{1}s}g\left(\frac{x(t-s)v(t-s)y_{1}}{x_{2}v_{2}y(t)}\right)ds\\
				&\enspace-\frac{\beta_{2}x_{2}c_{2}}{A_{1}}\int_{0}^{\infty}f_{1}(s)e^{-m_{1}s}g\left(\frac{x(t-s)c(t-s)y_{2}}{x_{2}c_{2}y(t)}\right)ds\\
				&\enspace-\frac{\beta_{1}x_{2}v_{2}}{A_{2}}\int_{0}^{\infty}f_{2}(s)e^{-m_{2}s}g\left(\frac{y(t-s)v_{2}}{v(t)y_{2}}\right)ds\\
				&\enspace-\beta_{2}x_{2}c_{2}g\left(\frac{c_{2}y(t)}{c(t)y_{2}}\right)-\frac{py(t)}{A_{1}(h+z(t))}(z(t)-z_{2})^{2}.
			\end{aligned}
		\end{equation}		
		Using the same method as in Theorem~\ref{theorem2}, and noting that $\mathcal{R}_{0}, \mathcal{R}_{1} > 1$, we have $\frac{dL_{2}(t)}{dt} \leq 0$. Furthermore, $\frac{dL_{2}(t)}{dt} = 0$ if and only if $x(t) = x_{2}$, $y(t) = y_{2}$, $c(t) = c_{2}$, $v(t) = v_{2}$, and $z(t) = z_{2}$. The set $\mathcal{M}_{2} = \{E_{2}\}$ represents the largest invariant subset of $\left\{(x(t), y(t), c(t), v(t), z(t)) \mid \frac{dL_{2}(t)}{dt} = 0\right\}$. Following LaSalle's invariance principle, it follows that if $\mathcal{R}_{0}, \mathcal{R}_{1} > 1$, then $E_{2}$ is globally asymptotically stable.

	\end{proof}
	\section{Numerical simulation}

	In this section, we investigate the global stability through numerical simulations. For simplicity, we reformulate the system given in equation~(\ref{eq2}) by introducing specific distribution functions as in  \cite{YANG2015183}:
	\begin{equation}
		\begin{aligned}
			f_{i}(s) = \delta(s - \tau_{i}^{\prime}), \quad i = 1, 2,
		\end{aligned}
	\end{equation}
	where $\delta(\cdot)$ represents the Dirac delta function.

	By exploiting the properties of the Dirac delta function, we derive the following expressions:
	\begin{equation}
		\begin{aligned}
			A_{i} = \int_{0}^{\infty} \delta(s - \tau_{i}^{\prime}) e^{-m_{i}s}ds& =  e^{-m_{i}\tau_{i}^{\prime}}, \quad i = 1, 2, \\
			\int_{0}^{\infty} \delta(s - \tau_{i}^{\prime}) e^{-m_{i}s} \varphi(t - s)ds &= e^{-m_{i}\tau_{i}^{\prime}} \varphi(t - \tau_{i}^{\prime}), \quad i = 1, 2.
		\end{aligned}
	\end{equation}
	
	With these expressions, the following values can be computed:
	\begin{equation}
		\begin{aligned}
			\mathcal{R}_{0} = \frac{\beta_{1} e^{-m_{1} \tau_{1}^{\prime}} k e^{-m_{2} \tau_{2}^{\prime}} d_{3} x_{0} + \beta_{2} e^{-m_{1} \tau_{1}^{\prime}} \alpha_{2} x_{0} d_{4}}{d_{3} (\alpha_{1} + d_{2})}, \;
			\mathcal{R}_{1} = \frac{c d_{1} d_{3} d_{4} (\mathcal{R}_{0} - 1)}{h d_{5} (\beta_{1} k e^{-m_{2} \tau_{2}^{\prime}} d_{3} + \beta_{2} \alpha_{2} d_{4})}.
		\end{aligned}
	\end{equation} 
	
	Appropriate parameters are selected and listed in Table~\ref{table}.

	\begin{table}[H]
		\begin{center}
			\caption{Parameters used in numberical simulation.}
			\small
			\begin{adjustbox}{height=1.0cm, width=13.5cm} 
				\begin{tabular}{ccccccccccccccccccc}\hline
					P & $\lambda$ & $\beta_{1}$ & $\beta_{2}$ & $d_{1}$ & $m_{1}$ & $\alpha_{1}$ & $d_{2}$ & $p$ & $\alpha_{2}$ & $d_{3}$ & $k$ & $m_{2}$ & $d_{4}$ & $c$ & $h$ & $d_{5}$ \\ \hline
					$E_{0}$ & 1 & 0.004 & 0.005 & 0.2 & 0.3  & 0.1 & 0.25 & 0.2 & 0.1 & 0.25 & 8 & 0.3 & 0.25 & 0.3 & 0.01 & 0.25 \\ 
					$E_{1}$ & 20 & 0.004 & 0.005 & 0.2 & 0.3 & 0.1 & 0.25 & 0.02 & 0.6 & 0.25 & 20 & 0.3 & 0.25 & 0.003 & 0.03 & 0.25\\
					$E_{2}$ & 20 & 0.004 & 0.005 & 0.2 & 0.3 & 0.1 & 0.25 & 0.02 & 0.6 & 0.25 & 20 & 0.3 & 0.25 & 0.03 & 0.03 & 0.25\\\hline
				\end{tabular}
				\label{table}
			\end{adjustbox}
		\end{center}
	\end{table}

	Under the parameter conditions specified in the table, we employ numerical simulations to discuss the results in Section 3.
	
	We initially assess the global stability of $E_0 = \left(\frac{\lambda}{d_1}, 0, 0, 0, 0\right)$. Different initial values and time delays are presented as follows:
	\begin{equation}
		\begin{aligned}
			&\Phi_1 = (5, 5, 6, 3, 3.5), \quad \Phi_2 = (6, 2, 7, 2, 4.5), \quad \Phi_3 = (4, 3, 8, 4, 4), \\
			&\text{lags1} = (5, 3), \quad \text{lags2} = (5, 2), \quad \text{lags3} = (2, 3).
		\end{aligned}
	\end{equation}
	The simulation results indicate that, regardless of differing time delays, the dynamics consistently converge to $E_0$, thereby validating the assertions of Theorem~\ref{theorem1}. Specifically, with the time delays \text{lags1} = (5, 3) and the initial values provided, the reproduction number $\mathcal{R}_0 = 0.2097 < 1$ is maintained. The corresponding simulation outcomes are depicted in Figure~\ref{Fig1}. It is observed that the global stability of $E_0 = \left(\frac{\lambda}{d_1}, 0, 0, 0, 0\right)$ does not depend on the initial values. From $\Phi_1 = (5, 5, 6, 3, 3.5)$ and the specified time delays, under the consistent condition that $\mathcal{R}_0 < 1$, further simulation results are shown in Figure~\ref{Fig2}.

	\begin{figure}[H]
		\centering
		\begin{minipage}[t]{1.0\linewidth}
			\centering
			\begin{tabular}{@{\extracolsep{\fill}}c@{}c@{}c@{}@{\extracolsep{\fill}}}
				\includegraphics[width=0.33\linewidth]{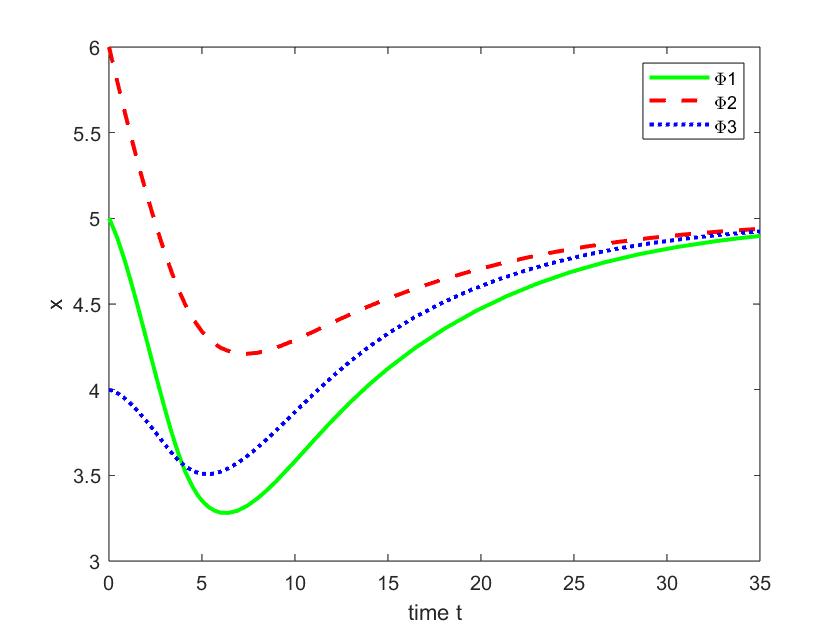} &
				\includegraphics[width=0.33\linewidth]{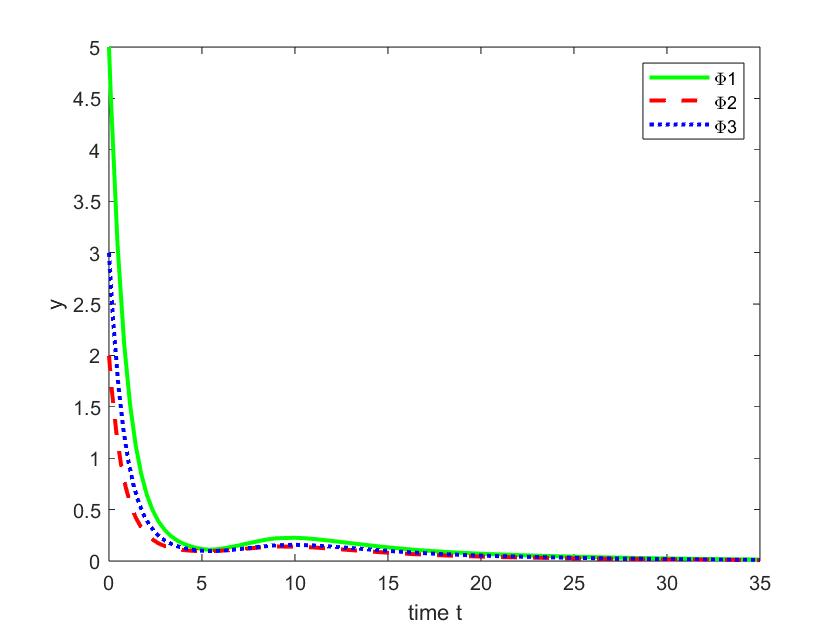}&
				\includegraphics[width=0.33\linewidth]{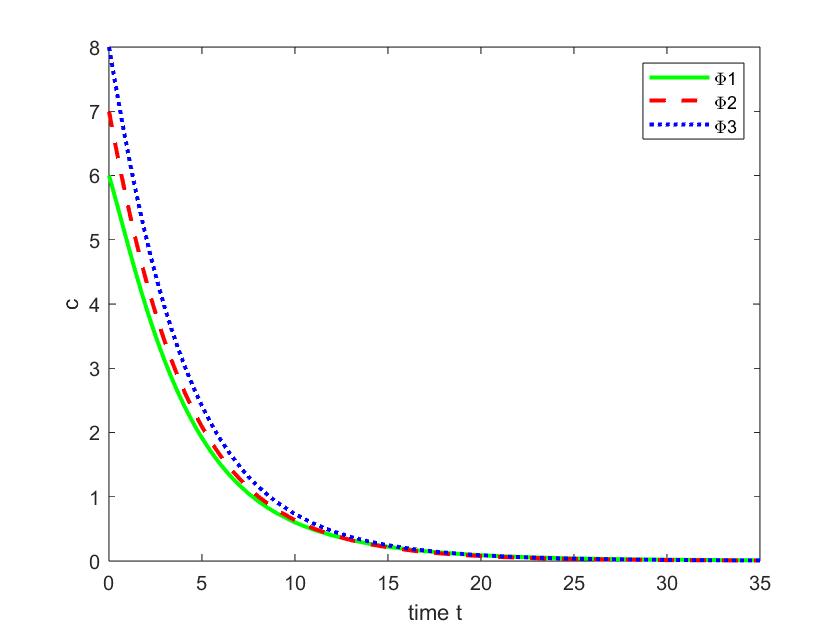}\\
			\end{tabular}
		\end{minipage}
		\begin{minipage}[t]{1.0\linewidth}
			\centering
			\begin{tabular}{@{\extracolsep{\fill}}c@{}c@{}@{\extracolsep{\fill}}}
				\includegraphics[width=0.33\linewidth]{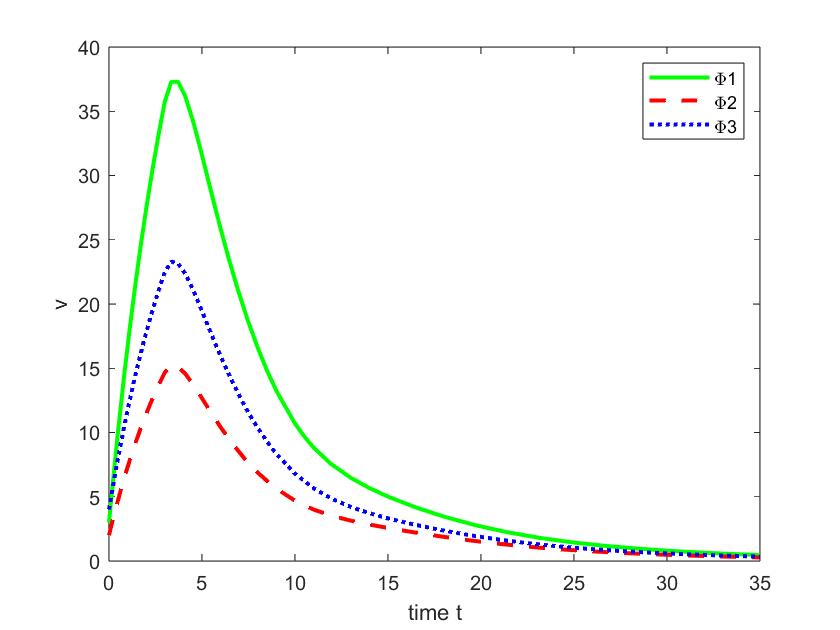} &
				\includegraphics[width=0.33\linewidth]{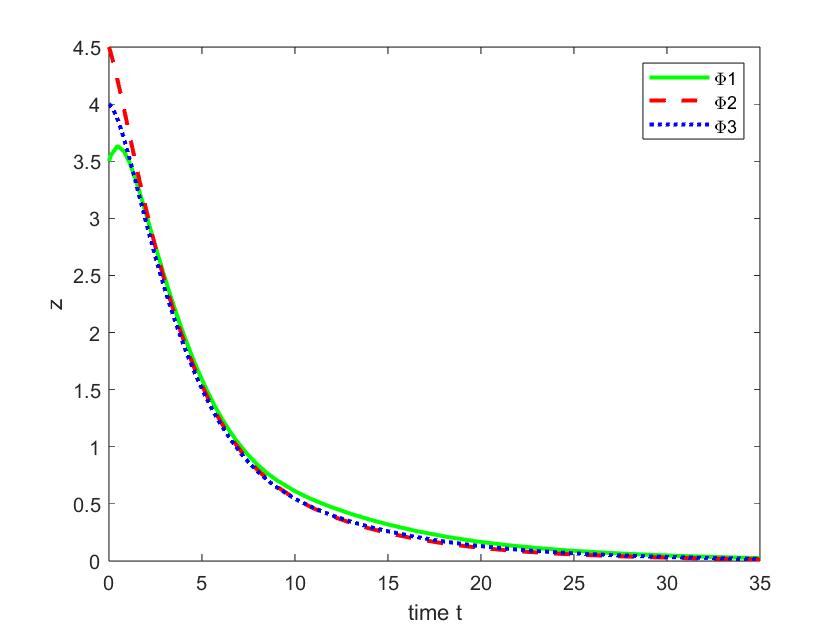}\\
			\end{tabular}
		\end{minipage}
		\caption{The simulation results of $E_{0}$ with different initial values}
		\label{Fig1}
	\end{figure}
	\begin{figure}[H]
		\centering
		\begin{minipage}[t]{1.0\linewidth}
			\centering
			\begin{tabular}{@{\extracolsep{\fill}}c@{}c@{}c@{}@{\extracolsep{\fill}}}
				\includegraphics[width=0.33\linewidth]{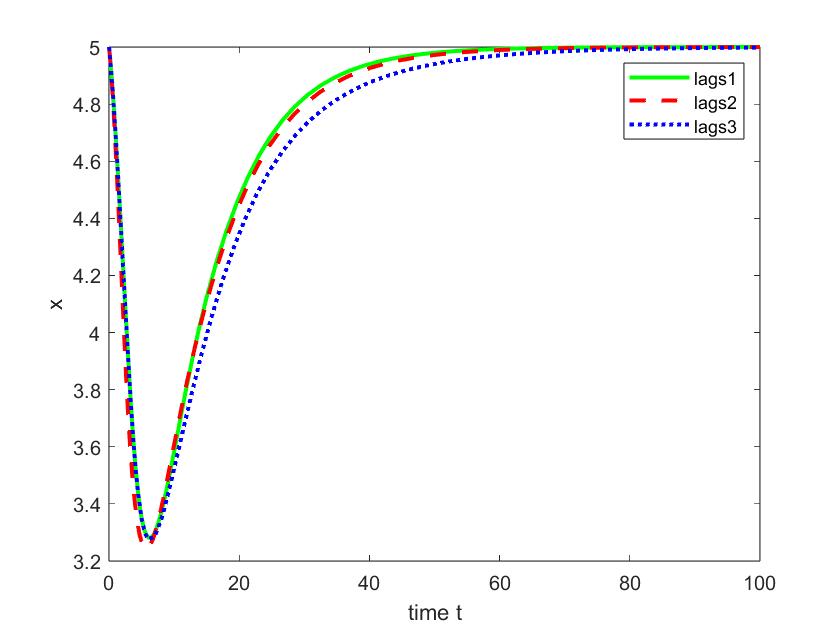} &
				\includegraphics[width=0.33\linewidth]{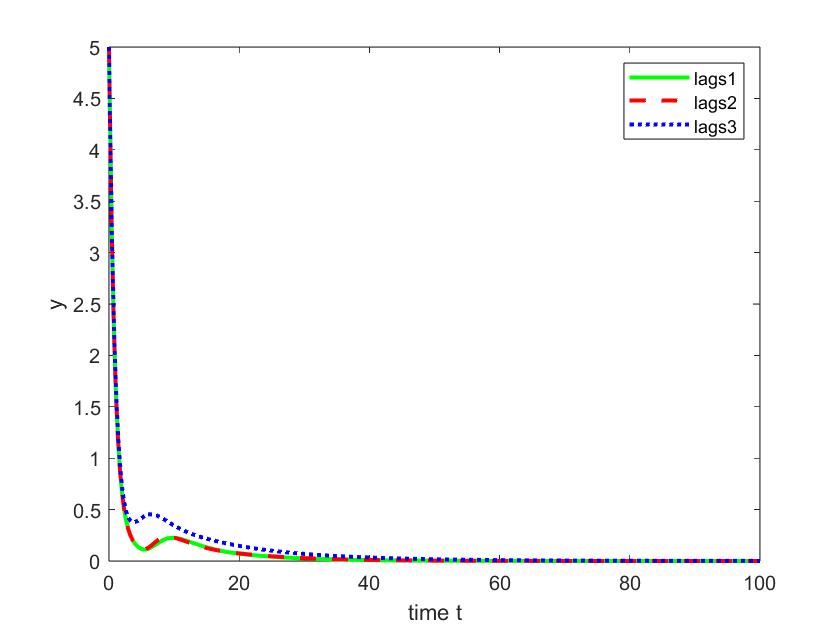}&
				\includegraphics[width=0.33\linewidth]{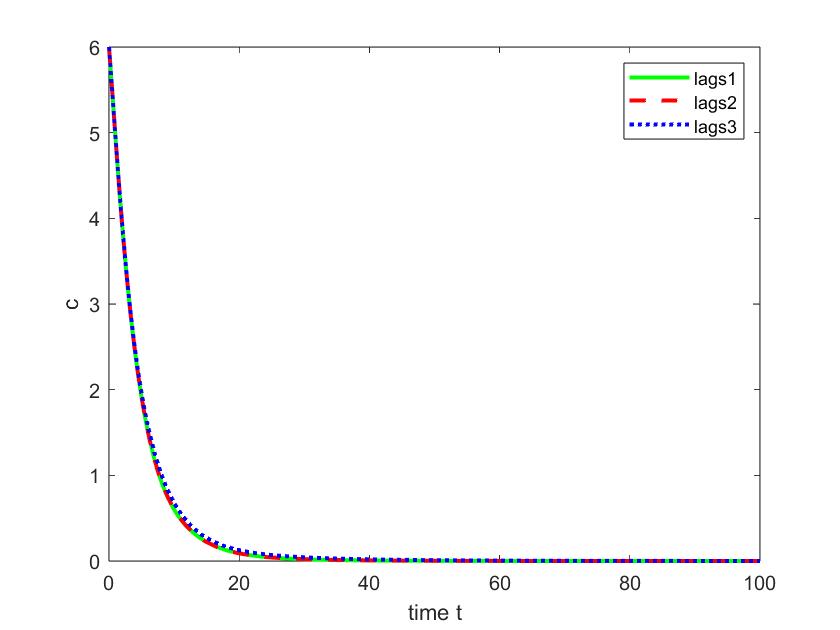}\\
			\end{tabular}
		\end{minipage}
		\begin{minipage}[t]{1.0\linewidth}
			\centering
			\begin{tabular}{@{\extracolsep{\fill}}c@{}c@{}@{\extracolsep{\fill}}}
				\includegraphics[width=0.33\linewidth]{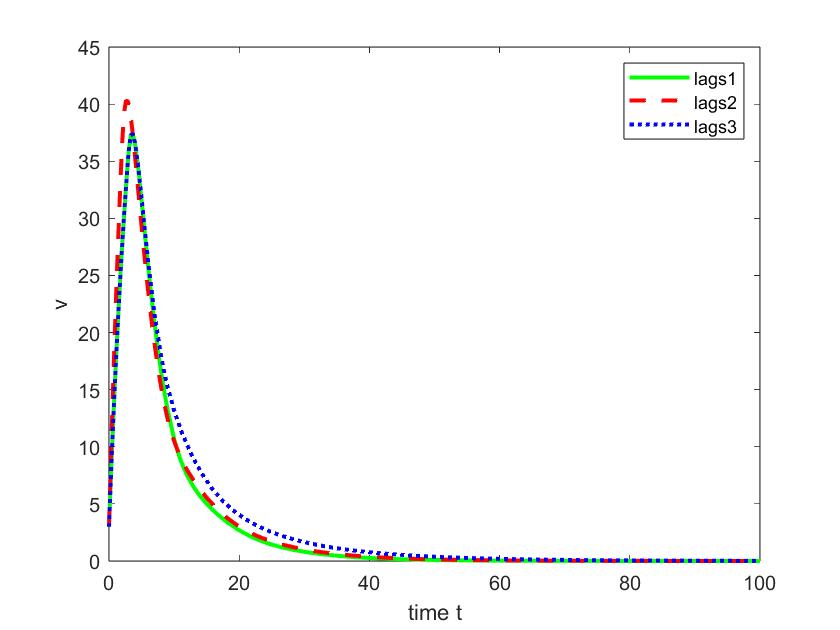} &
				\includegraphics[width=0.33\linewidth]{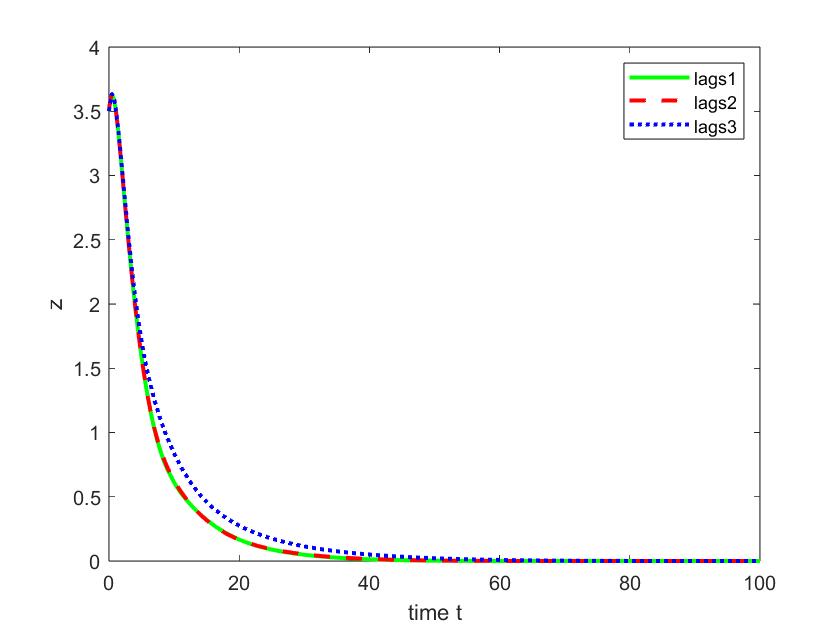}\\
			\end{tabular}
		\end{minipage}
		\caption{The simulation results of $E_{0}$ with different time delays}
		\label{Fig2}
	\end{figure}

	We examine the global stability of $E_1$. The various initial values and time delays are delineated as follows:
    \begin{equation}
	\begin{aligned}
		&\Phi_1 = (5, 5, 6, 3, 35), \quad \Phi_2 = (6, 2, 7, 2, 45), \quad \Phi_3 = (4, 3, 8, 4, 25), \\
		&\text{lags1} = (5, 3), \quad \text{lags2} = (5, 2), \quad \text{lags3} = (4,7).
	\end{aligned}
    \end{equation}
	In the scenario with \text{lags1} = (5, 3) and the specified initial values, the reproduction numbers $\mathcal{R}_0 = 2.2648 > 1$ and $\mathcal{R}_1 = 0.7121 < 1$ are maintained. The simulation outcomes are presented in Figure~\ref{Fig3}. It is evident that the global stability of $E_1$ is influenced by the initial values. Furthermore, with $\Phi_1 = (5, 5, 6, 3, 35)$ and the given time delays, the conditions $\mathcal{R}_0 > 1$ and $\mathcal{R}_1 < 1$ consistently hold. The corresponding simulation results are depicted in Figure~\ref{Fig4}.

	\begin{figure}[H]
		\centering
		\begin{minipage}[t]{1.0\linewidth}
			\centering
			\begin{tabular}{@{\extracolsep{\fill}}c@{}c@{}c@{}@{\extracolsep{\fill}}}
				\includegraphics[width=0.33\linewidth]{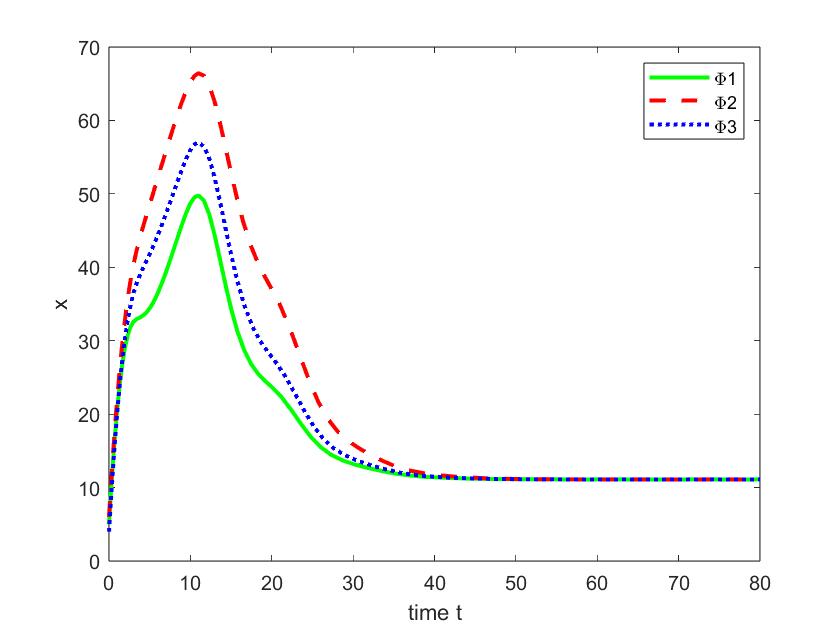} &
				\includegraphics[width=0.33\linewidth]{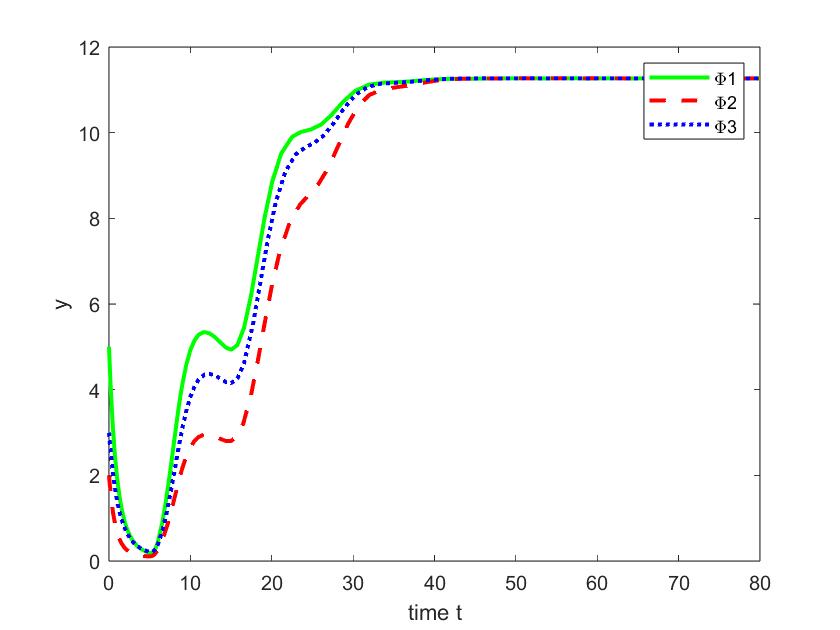}&
				\includegraphics[width=0.33\linewidth]{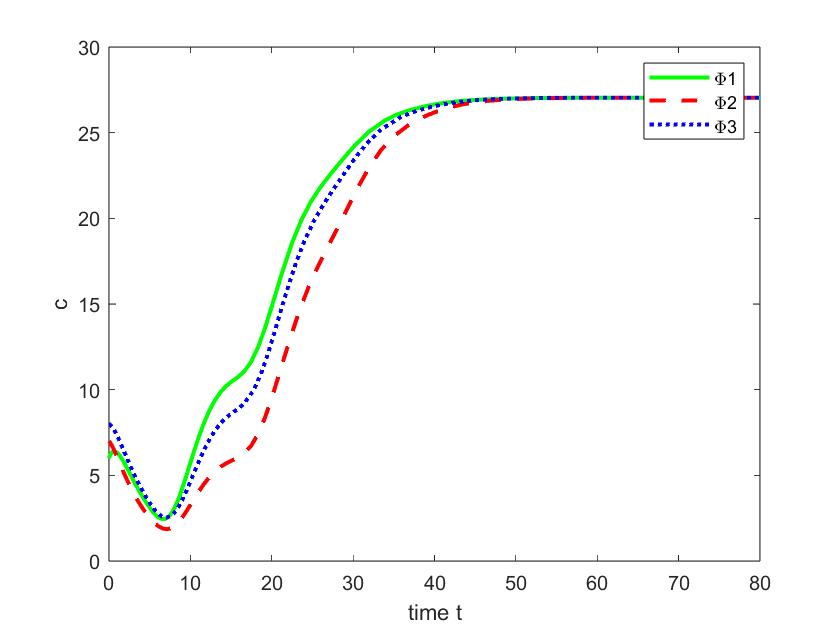}\\
			\end{tabular}
		\end{minipage}
		\begin{minipage}[t]{1.0\linewidth}
			\centering
			\begin{tabular}{@{\extracolsep{\fill}}c@{}c@{}@{\extracolsep{\fill}}}
				\includegraphics[width=0.33\linewidth]{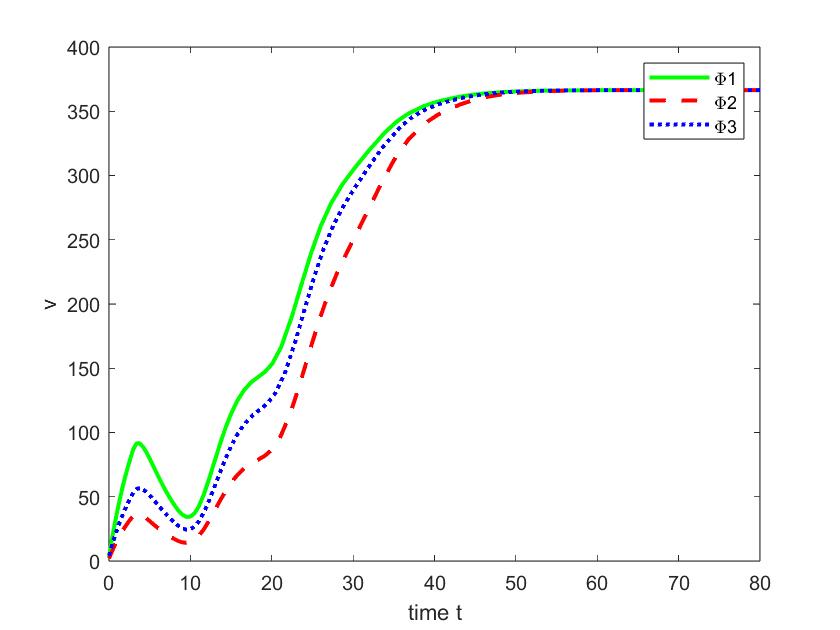} &
				\includegraphics[width=0.33\linewidth]{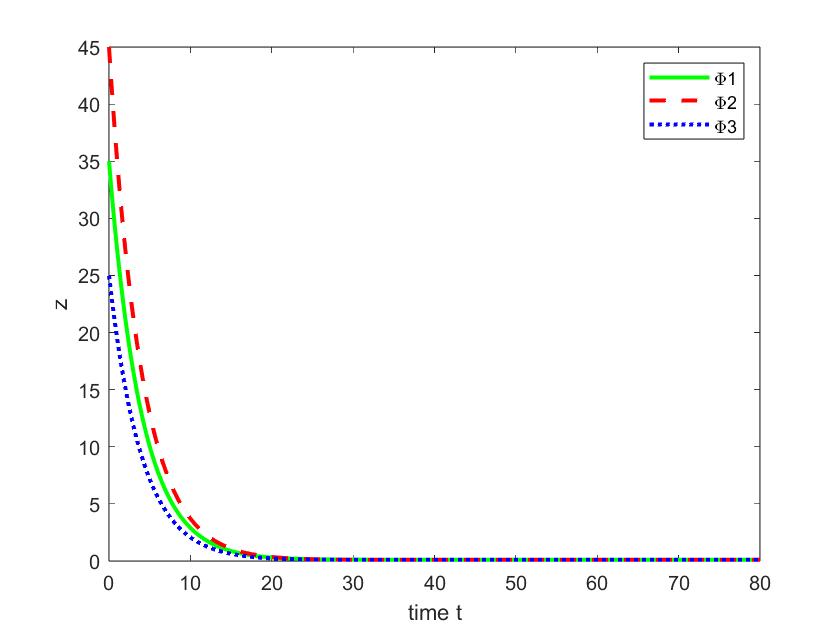}\\
			\end{tabular}
		\end{minipage}
		\caption{The simulation results of $E_{1}$ with different initial values}
		\label{Fig3}
	\end{figure}
	\begin{figure}[H]
		\centering
		\begin{minipage}[t]{1.0\linewidth}
			\centering
			\begin{tabular}{@{\extracolsep{\fill}}c@{}c@{}c@{}@{\extracolsep{\fill}}}
				\includegraphics[width=0.33\linewidth]{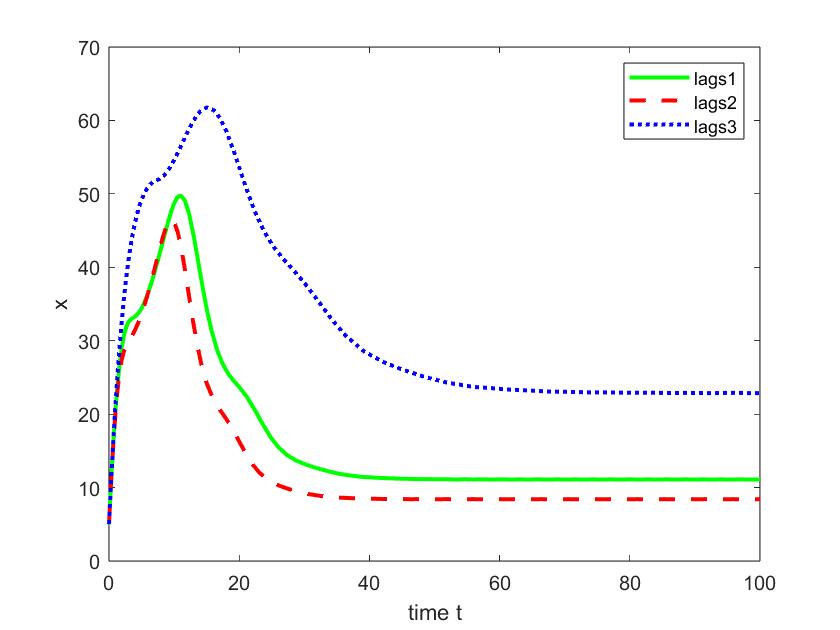} &
				
				\includegraphics[width=0.33\linewidth]{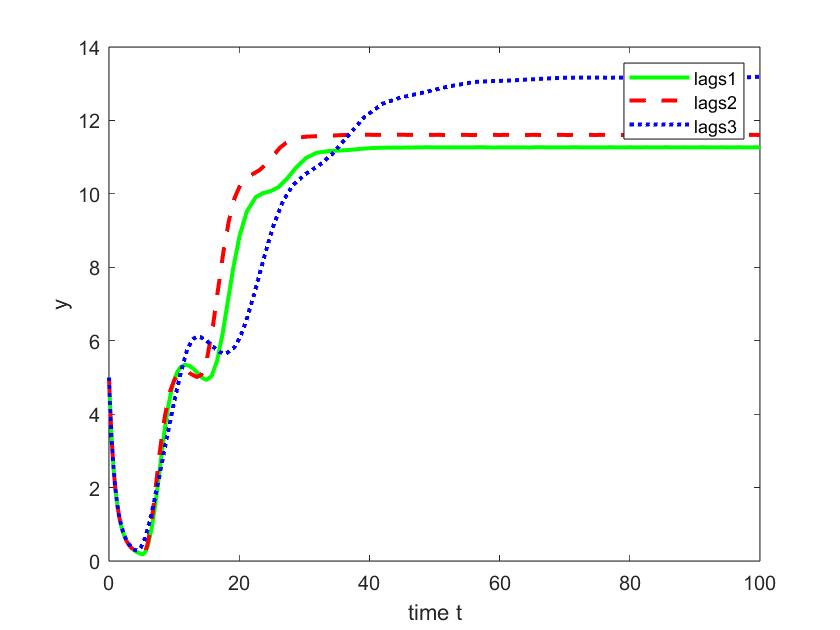}&
				\includegraphics[width=0.33\linewidth]{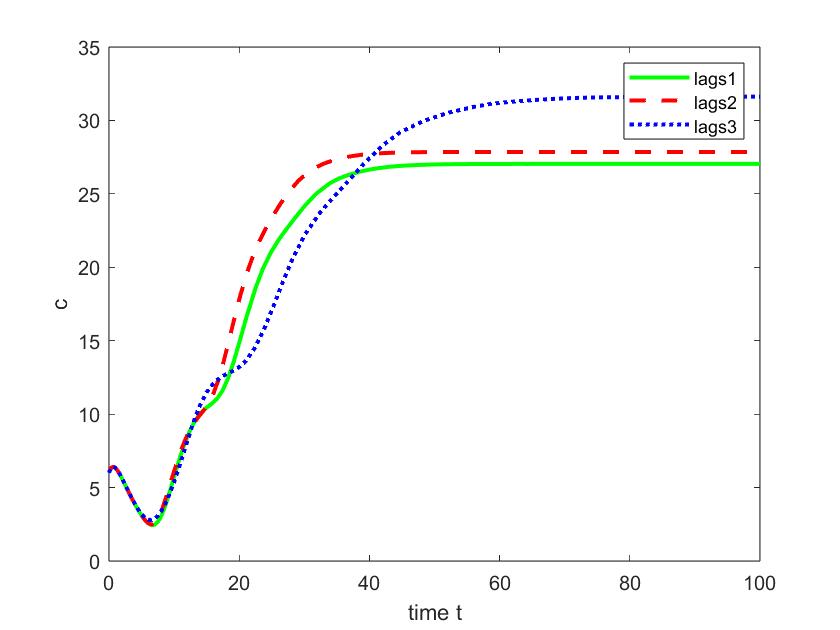}\\
			\end{tabular}
		\end{minipage}
		\begin{minipage}[t]{1.0\linewidth}
			\centering
			\begin{tabular}{@{\extracolsep{\fill}}c@{}c@{}@{\extracolsep{\fill}}}
				\includegraphics[width=0.33\linewidth]{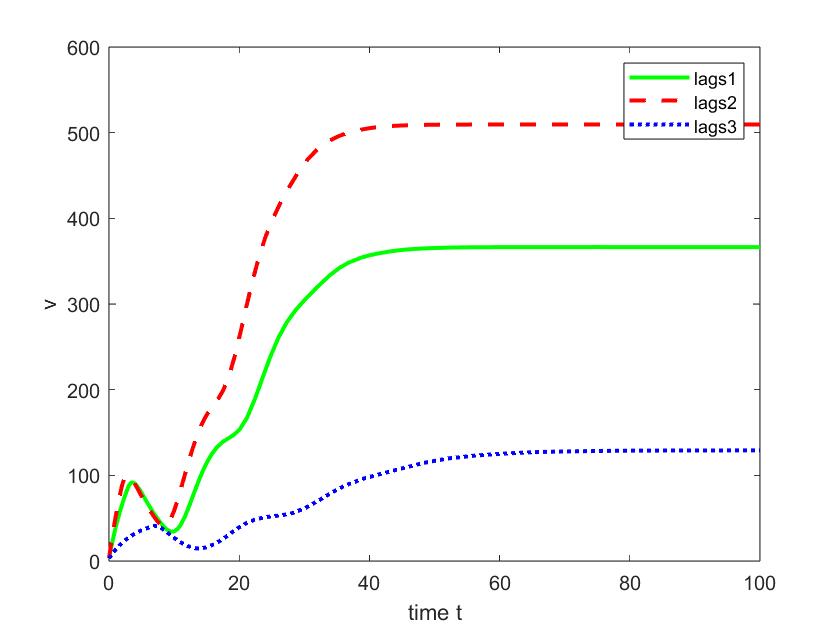} &
				\includegraphics[width=0.33\linewidth]{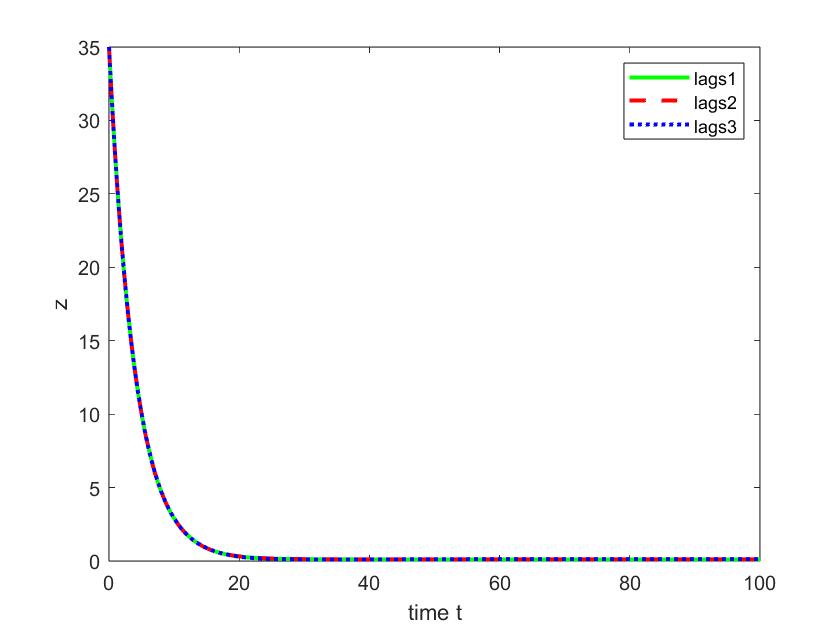}\\
			\end{tabular}
		\end{minipage}
		\caption{The simulation results of $E_{1}$ with different time delays}
		\label{Fig4}
	\end{figure}

	The simulation results corroborate Theorem~\ref{theorem2}. Additionally, it is observed that under varying time delays, the eventual infection state can differ even when the initial conditions are consistent. This variation highlights the significant role that time delay plays in determining the final infection status. Such dynamics reflect the differential impact of the virus on diverse populations, underscoring the biological complexities involved.

	Under the condition that $\tau_{1}^{\prime}, \tau_{2}^{\prime} \geq 0$, we discuss the global stability of $E_{2}$. The various initial values and time delays are specified as follows:
    \begin{equation}
	\begin{aligned}
		&\Phi_{1} = (12, 4, 35, 1, 10), \quad \Phi_{2} = (25, 3, 40, 2, 15), \quad \Phi_{3} = (40, 10, 25, 4, 13), \\
		&\text{lags1} = (5, 4), \quad \text{lags2} = (5, 2), \quad \text{lags3} = (2, 4).
	\end{aligned}
    \end{equation}
	For the case with \text{lags1} = (5, 4) and the given initial values, the reproduction numbers $\mathcal{R}_0 = 1.7274 > 1$ and $\mathcal{R}_1 = 5.3690 > 1$ are observed. The corresponding simulation outcomes are shown in Figure~\ref{Fig5}. It is evident that the global stability of $E_{2}$ does not depend on the initial values. With $\Phi_{1} = (12, 4, 35, 1, 10)$ and the specified time delays, the conditions $\mathcal{R}_0 > 1$ and $\mathcal{R}_1 > 1$ consistently hold. The results are displayed in Figure~\ref{Fig6}. The simulations suggest that larger values of $\tau_{1}^{\prime}$ and $\tau_{2}^{\prime}$ lead to a more favorable final infection state for the human body. Specifically, the larger the parameter $\tau_{1}^{\prime}$, the more beneficial the outcome. Therefore, using drugs to extend the time delays, particularly $\tau_{1}^{\prime}$, could potentially offer a more effective biological intervention.

	\begin{figure}[H]
		\centering
		\begin{minipage}[t]{1.0\linewidth}
			\centering
			\begin{tabular}{@{\extracolsep{\fill}}c@{}c@{}c@{}@{\extracolsep{\fill}}}
				\includegraphics[width=0.33\linewidth]{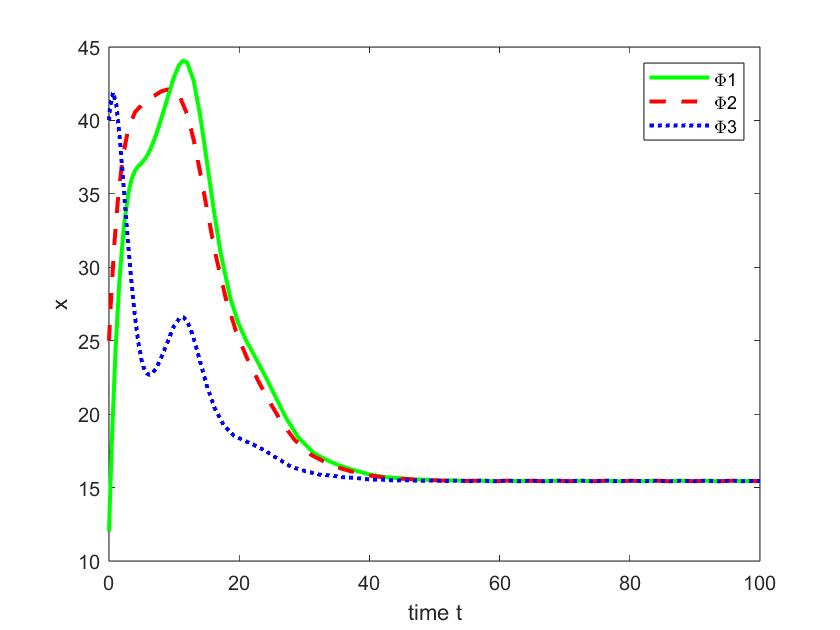} &
				\includegraphics[width=0.33\linewidth]{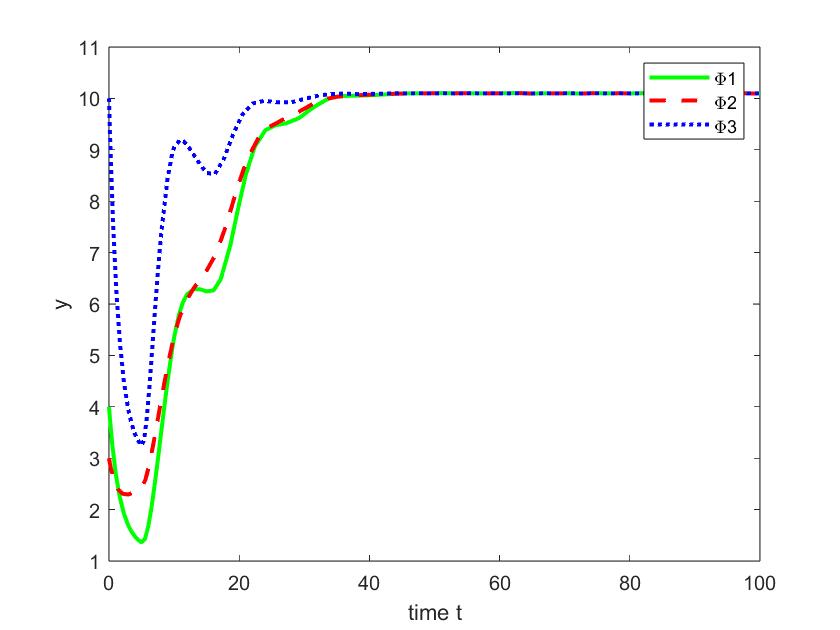}&
				\includegraphics[width=0.33\linewidth]{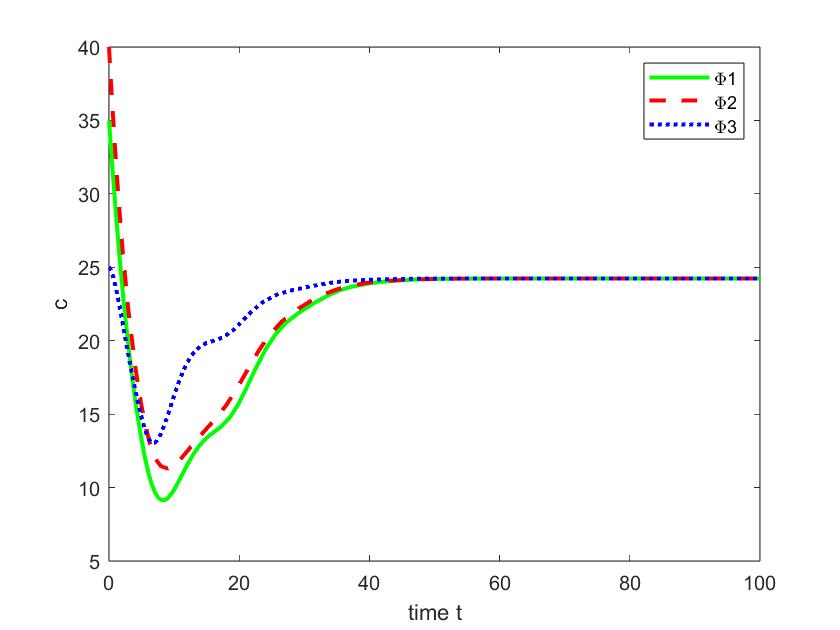}\\
			\end{tabular}
		\end{minipage}
		\begin{minipage}[t]{1.0\linewidth}
			\centering
			\begin{tabular}{@{\extracolsep{\fill}}c@{}c@{}@{\extracolsep{\fill}}}
				\includegraphics[width=0.33\linewidth]{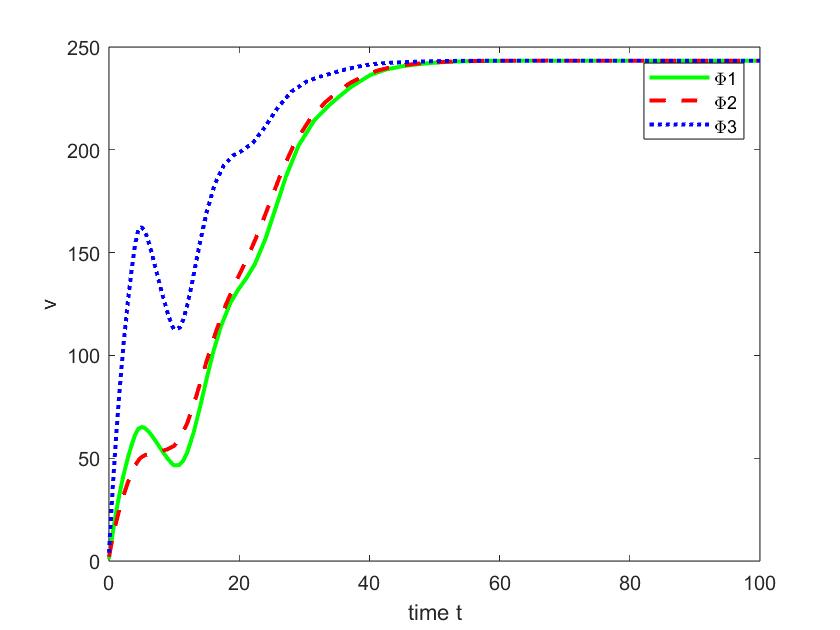} &
				\includegraphics[width=0.33\linewidth]{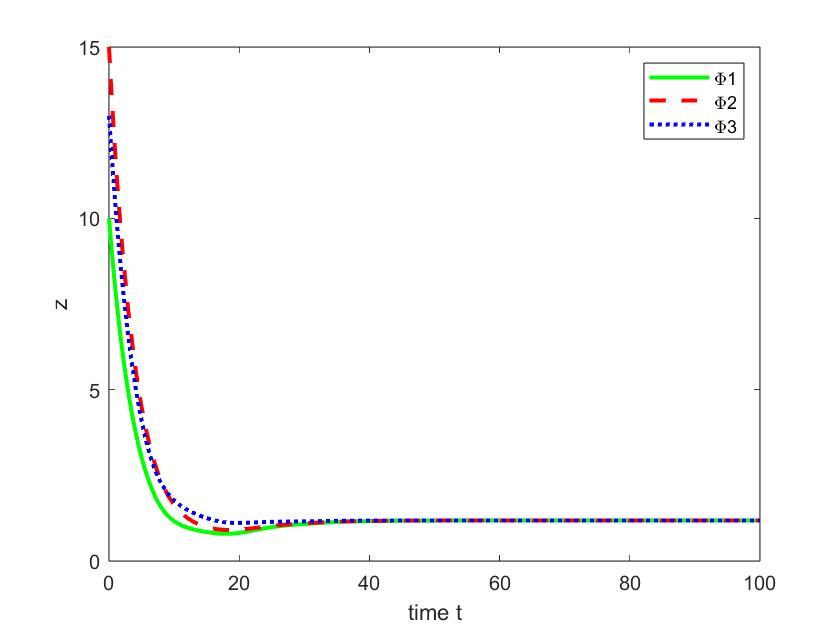}\\
			\end{tabular}
		\end{minipage}
		\caption{The simulation results of $E_{2}$ with different initial values}
		\label{Fig5}
	\end{figure}
	\begin{figure}[H]
		\centering
		\begin{minipage}[t]{1.0\linewidth}
			\centering
			\begin{tabular}{@{\extracolsep{\fill}}c@{}c@{}c@{}@{\extracolsep{\fill}}}
				\includegraphics[width=0.33\linewidth]{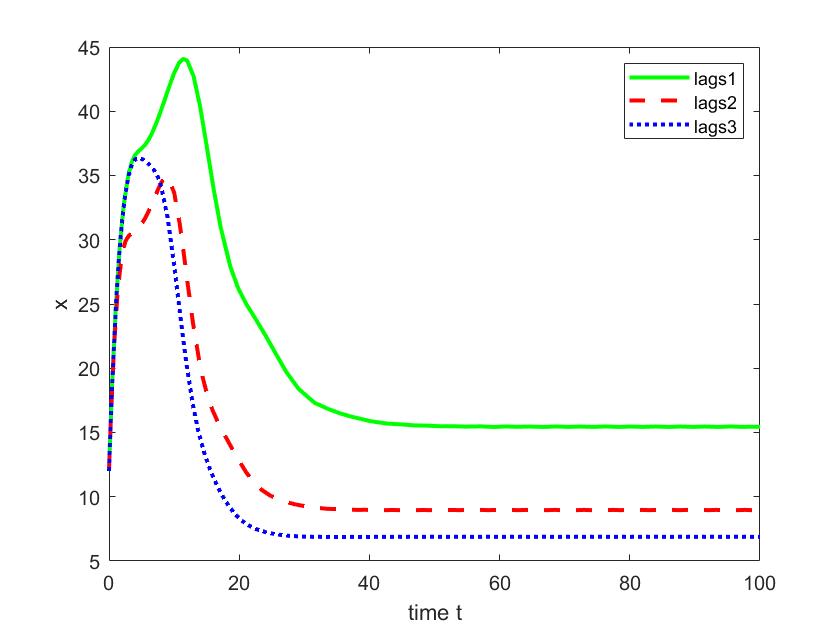} &
				\includegraphics[width=0.33\linewidth]{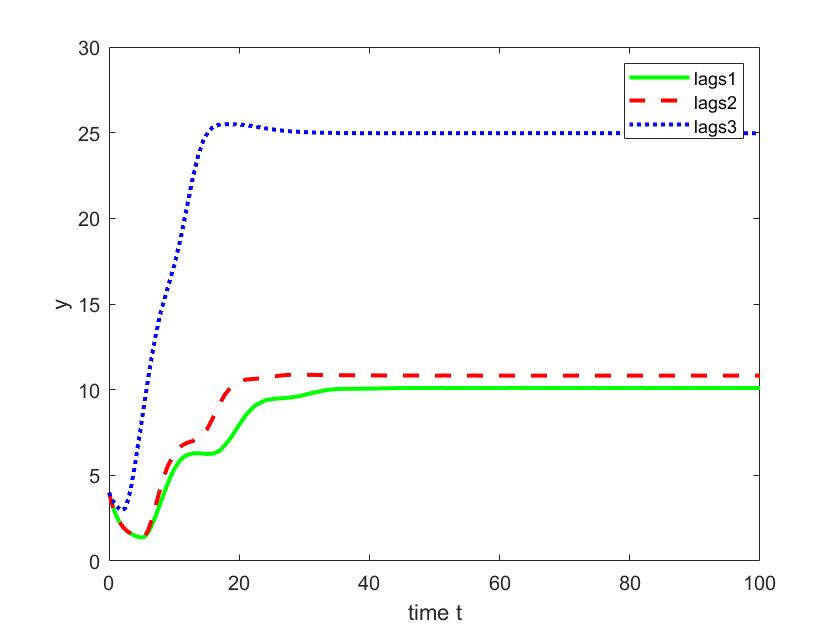}&
				\includegraphics[width=0.33\linewidth]{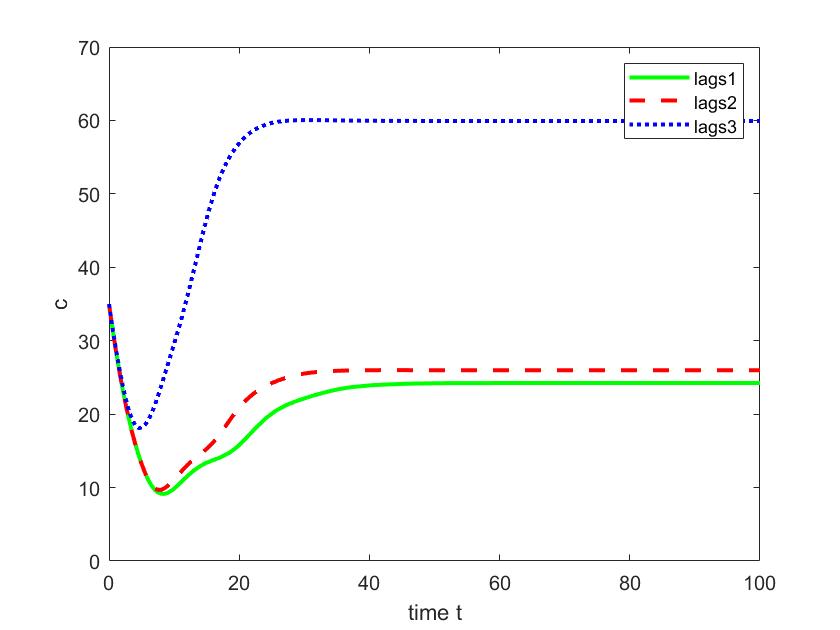}\\
			\end{tabular}
		\end{minipage}
		\begin{minipage}[t]{1.0\linewidth}
			\centering
			\begin{tabular}{@{\extracolsep{\fill}}c@{}c@{}@{\extracolsep{\fill}}}
				\includegraphics[width=0.33\linewidth]{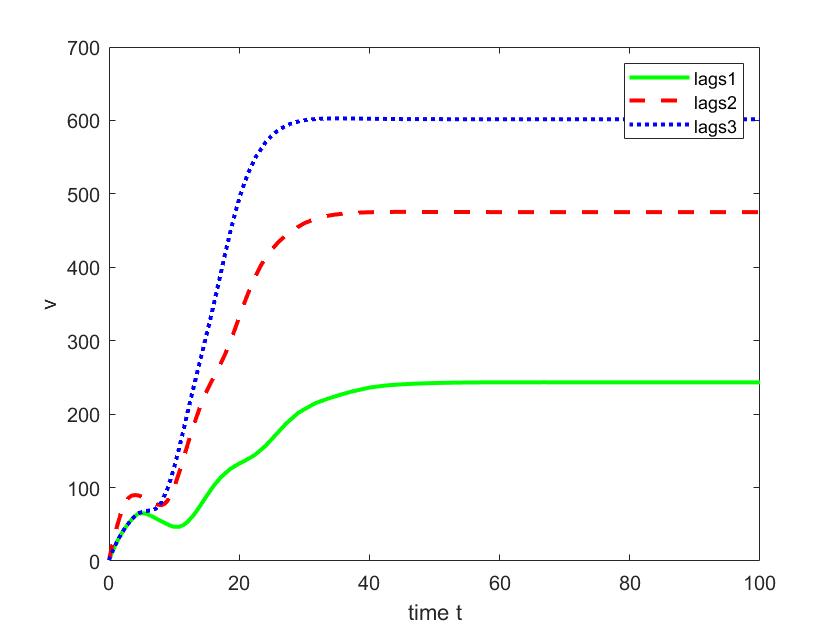} &
				\includegraphics[width=0.33\linewidth]{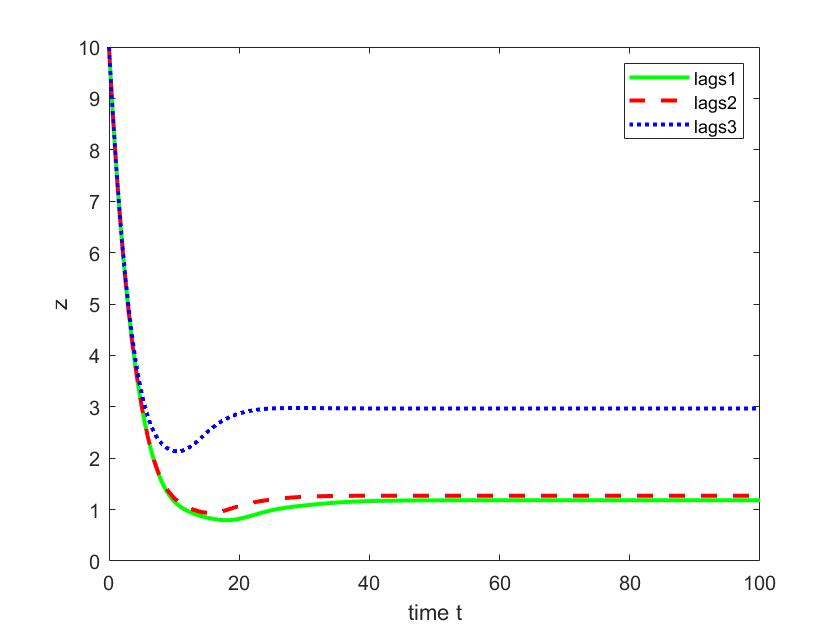}\\
			\end{tabular}
		\end{minipage}
		\caption{The simulation results of $E_{2}$ with different time delays}
		\label{Fig6}
	\end{figure}
	
	\section{Conclusion}

	In this paper, we have extended an HIV model with inflammatory cytokines by incorporating distributed delays and a saturated infection rate. We proceed to calculate the basic reproduction numbers and three equilibrium points: $E_0$, $E_1$, and $E_2$. We demonstrate that the convex cone  $K_5$ is invariant in relation to the system. Employing Lyapunov functions and LaSalle's invariance principle, we discuss the global stability of $E_1$, $E_2$, and $E_3$ under specific conditions. Our numerical simulations not only confirm the conclusions of the theorems but also suggest that the impact of virus infection can be mitigated by prolonging the time delays.

\vskip 20 pt
\noindent{\bf Acknowledgement}

The first author acknowledges the support from the Simons Foundation (\#585201).  The second author  is  partially  supported by  the National Key Rearch and Development Program of China 2020YFA0713100 and by the National Natural Science Foundation of China (Grant No. 11721101).

	\bibliographystyle{plain}
	\bibliography{k.bib}		
	
\end{document}